\documentclass
[12pt]
{amsart}

\usepackage{hyperref}
\usepackage{cite}
\usepackage{amssymb}
\usepackage{amsmath}
\usepackage{amsthm}
\usepackage{amsfonts}
\usepackage{graphicx}
\usepackage{color}
\usepackage[usenames,dvipsnames,svgnames,table]{xcolor}
\usepackage[toc,page]{appendix}
\usepackage{bm,bbm,cancel}
\usepackage[normalem]{ulem}
\usepackage{mathtools}
\usepackage{comment}
\usepackage{arydshln}
\usepackage[shortlabels]{enumitem}
\usepackage{soul}
\setstcolor{red}

\hypersetup{
linkcolor=blue,
filecolor=magenta,
urlcolor=cyan,
citecolor=blue
}
\usepackage[usenames,dvipsnames]{xcolor}
\usepackage[textwidth=6in]{todonotes}
\newcommand{\beq}{\begin{equation}}
\newcommand{\eeq}{\end{equation}}
\newcommand{\bea}{\begin{eqnarray}}
\newcommand{\eea}{\end{eqnarray}}
\newcommand{\beas}{\begin{eqnarray*}}
\newcommand{\eeas}{\end{eqnarray*}}

\newtheorem{theorem}{Theorem}[section]
\newtheorem*{theorem*}{Theorem}

\newtheorem{prop}[theorem]{Proposition}
\newtheorem{corollary}[theorem]{Corollary}
\newtheorem{lemma}[theorem]{Lemma}
\newtheorem{remark}[theorem]{Remark}
\newtheorem{example}[theorem]{Example}

\theoremstyle{definition}
\newtheorem{definition}[theorem]{Definition}
\newtheorem{notation}[theorem]{Notation}

\numberwithin{equation}{section}

\newcommand{\R}{\mathbb R}

\newcommand{\p}{\mathbb P}

\newcommand{\N}{\mathbb N}

\newcommand{\E}{\mathbb E}

\newcommand{\h}{\mathcal H}

\newcommand{\dH}{{d_1}}
\newcommand{\X}{\mathcal X}
\newcommand{\Y}{\mathcal Y}

\DeclareMathOperator{\ad}{ad}

\newcommand\rw{D_{\frac{1}{n}}S_n}
\newcommand\drw{\Psi_m(Y_n^m)}

\title[Sub-Riemannian LDP]{Large deviations principle for sub-Riemannian random walks}

\author[Gordina]{Maria Gordina{$^{\dag }$}}
\thanks{\footnotemark {$\dag$} Research was supported in part by NSF Grant DMS-2246549.}
\address{ Department of Mathematics\\
University of Connecticut\\
Storrs, CT 06269,  U.S.A.}
\email{maria.gordina@uconn.edu}

\author[Melcher]{Tai Melcher{$^{\dag\dag }$}}
\thanks{\footnotemark {$\dag\dag$} Research was supported in part by NSF Grant DMS-1255574.}
\address{Department of Mathematics \\
University of Virginia\\
Charlottesville, VA 22904, U.S.A.}
\email{melcher@virginia.edu}

\author[Mikulincer]{Dan Mikulincer{$^{\ddag}$}}
\thanks{\footnotemark {$\ddag$} Research was supported in part by Simons Investigator award (622132, PI- Mossel)}
\address{Department of Mathematics \\
Massachusetts Institute of Technology\\
Cambridge, MA 02139, U.S.A.}
\email{danmiku@mit.edu}

\author[Wang]{Jing Wang{$^{\mathsection}$}}
\thanks{\footnotemark {$\mathsection$} Research was supported in part by  NSF Grant DMS-1855523 and DMS-2246817}
\address{ Department of Mathematics\\
Purdue University\\
West Lafayette, IN 47907,  U.S.A.}
\email{jingwang@purdue.edu}



\begin{document}

\begin{abstract}
We study large deviations for random walks on stratified (Carnot) Lie groups. For such groups, there is a natural collection of vectors which generates their Lie algebra, and we consider random walks with increments in only these directions. Under certain constraints on the distribution of the increments, we prove a large deviation principle for these random walks with a natural rate function adapted to the sub-Riemannian geometry of these spaces.
\end{abstract}

\maketitle

\tableofcontents

\section{Introduction}\label{s.intro}
The study of large deviations for random walks goes back to Cram\'{e}r's work in 1938. A general version of Cram\'{e}r's theorem says that, given an i.i.d.~sequence of random variables $Y_1, Y_{2}, \dots$ taking values in a locally convex vector Polish space, under suitable conditions on the tail distribution of the random variable $Y_1$, the random walks induced by $\{Y_i\}_{i =1}^{\infty}$ satisfy a Large Deviations Principle (LDP) with a good rate function.

Due to its generality, Cram\'{e}r's theorem later became a starting point of various classical LDP results on path spaces (see for instance \cite{DeuschelStroockBook1989, DemboZeitouniBook2010}). However, few results are known addressing the question of how such results depend on the underlying topological space, such as for a random walk on a Riemannian or a sub-Riemannian manifold. Recently Kraaij-Redig-Versendaal in \cite{KraaijRedigVersendaal2019} and Versendaal \cite{Versendaal2019a, Versendaal2019b} obtained a Cram\'{e}r-type LDP for geodesic random walks on Riemannian manifolds that were first introduced by E.~Jorgensen in \cite{JorgensenE1975}. These authors were also able to obtain related Moguslkii's theorem in this setting. Closer to our setting is a Cram\'{e}r-type LDP for random matrix products considered by C.~Sert in \cite{Sert2019}.

Our goal is to study Cram\'{e}r-type theorems in a sub-Riemannian setting, or more precisely on Carnot groups. For such groups the Lie algebras have a stratified structure that ensures H\"{o}rmander's condition is satisfied, and so induces a natural sub-Riemannian manifold structure. While we lack the structure of a vector space, the upshot is that the group structure affords natural constructions of random walks on these curved spaces equipped with a sub-Riemannian metric. For example, in  \cite{Pap1993, Pap1997a, Pap1999}~random walks have been constructed on nilpotent groups, while in \cite{GordinaLaetsch2017, AgrachevBoscainNeelRizzi2018, BoscainNeelRizzi2017} a very different construction has been introduced on more general sub-Riemannian manifolds. In this paper the random walks we consider shall rely on the Lie group structure of the underlying manifold. Such random walks were first considered by Pap in \cite{Pap1993} to prove a central limit theorem on nilpotent Lie groups (see also \cite{Benard2023}). Thus while the small deviations regime is well-understood and covered by central limit theorem type results, our aim is to cover the complementary large deviations regime.

To further motivate our results we note the growing interest in \emph{non-linear large deviations} in recent years (see \cite{ChatterjeeVaradhan2011, ChatterjeeDembo2016, LubetzkyZhao2015, Eldan2018, Augeri2020a} for some prominent examples). A key element in these works is that when the non-linear functional is highly symmetric, such as subgraph counts in Erd\H{o}s-R\'enyi graphs, precise and quantitative large deviations results can be obtained. As it will soon become apparent, our random walks can equivalently be regarded as non-linear, or even polynomial, functionals on appropriate product spaces. In this context, the non-linearity arises from the algebraic structure of the group operation. Furthermore, the group-theoretic aspects of these induced functionals entail numerous symmetries, which we can utilize to derive exact LDPs.

We now present our setting. Let $G$ be a Carnot group with Lie algebra $\mathfrak{g}$. That is, $G$ is a connected, simply-connected Lie group whose Lie algebra may be written as $\mathfrak{g}=\h\oplus\mathcal{V}$ such that any basis $\{\X_1,\ldots,\X_d\}$ of $\h$ satisfies H\"ormander's condition 
\begin{align*}
&
\operatorname{span}\{\X_{i_1},[\X_{i_1},\X_{i_{2}}],\ldots,[\X_{i_1},[\X_{i_{2}},[\cdots,\X_{i_r}]\cdots]]: i_{k}\in\{1,\ldots,d\}\} 
\\
&= \mathfrak{g} \text{ for some } r\in\mathbb{N}.
\end{align*}

We assume that $\h$ is equipped with an inner product $\langle\cdot,\cdot\rangle_\h$, and therefore the Carnot group $G$ has a natural sub-Riemannian structure. Namely, one may use left translation to define a \emph{horizontal distribution} $\mathcal{D}$, a sub-bundle of the tangent bundle $TG$, and a metric on $\mathcal{D}$ as follows. First, we identify the space $\mathcal{H} \subset \mathfrak{g}$ with $\mathcal{D}_{e}\subset T_eG$. Then for $g\in G$ let $L_g$ denote the left translation $L_gh :=gh$, and define $\mathcal{D}_{g}:=(L_g)_{\ast}\mathcal{D}_{e}$ for any $g \in G$. A metric on $\mathcal{D}$ may then be defined by translating back to $\mathcal{H}$,
\begin{align*}
\langle u, v\rangle_{\mathcal{D}_g} &:= \langle  (L_{g^{-1}})_{\ast} u,(L_{g^{-1}})_{\ast} v\rangle_{\mathcal{D}_e} \\
	&= \langle (L_{g^{-1}})_{\ast} u,(L_{g^{-1}})_{\ast} v\rangle_{\h}  \text{ for all } u, v\in\mathcal{D}_g.
\end{align*}
We will sometimes identify the horizontal distribution $\mathcal{D}$ with $\mathcal{H}$. Vectors in $\mathcal{D}$ are called \emph{horizontal}, and we say that a path $\gamma:[0,1]\to G$ is \emph{horizontal} if $\gamma$ is absolutely continuous and $\gamma^{\prime}(t)\in \mathcal{D}_{\gamma(t)}$ for a.e.~$t$. Equivalently, $\gamma$ is horizontal if the Maurer-Cartan form $c_{\gamma}\left( t \right):=( L_{\gamma\left( t \right)^{-1}})_{\ast} \gamma^{\prime} \left( t \right)\in \mathcal{H}$ for a.e. $t$. Such horizontal paths are used to define a left-invariant Carnot-Carath\'{e}odory distance $G$ as one of left-invariant homogeneous distances on $G$. We view the Carnot group $G$ as a metric space with respect to one of such a distance.

Suppose $\left\{ X_{n} \right\}_{n=1}^{\infty}$ is a sequence of i.i.d.~random variables with mean $0$ taking values in $\mathcal{H}$. We consider a \emph{sub-Riemannian random walk} on $G$ defined by
\begin{align*}
& S_0=e,
\\
& S_{n}:=\exp(X_1)\cdots \exp(X_{n}),\quad  n \in \mathbb{N}.
\end{align*}
One of the features of Carnot groups is a dilation which can be used to scale the random walk appropriately. We give more details in Section~\ref{s.carnot}, in particular, \eqref{e.Dilations} lists their main properties. For $a>0$, we denote by $D_{a}: G \longrightarrow G$ the dilation homomorphism on $G$ adapted to its stratified structure. If $\{X_n\} \subset \h$ are i.i.d.~with mean $0$, the law of large numbers says that almost surely
\[
\lim_{n\to \infty}D_{\frac{1}{n}} S_{n}=e;
\]
see for example \cite{Gaveau1977a}, or \cite{Neuenschwander1997}.

To state our results, we denote
\[
\Lambda(\lambda):= \Lambda_X(\lambda) :=\log \E[\exp(\langle\lambda,X_{k}\rangle_\h)]
\]
and let $\Lambda^*$ be the Legendre transform of $\Lambda$ defined by \eqref{e.FL}.
Let $\mu_n$ be the distribution of $D_{\frac{1}{n}} S_{n}$. Our main result is the following Cram\'{e}r-type large deviations principle for $\{\mu_n\}_{n=1}^\infty$.

\begin{theorem}\label{t.main}
Suppose $\{X_{k}\}_{k=1}^\infty$ are i.i.d. with mean $0$ random variables in $\h$ such that $ \Lambda_X(\lambda) $
exists for all $\lambda\in\h$. Further assume that one of the following assumptions is satisfied.
	\begin{enumerate}[(i)]
		\item $G$ is of step $2$ and the distribution of each $X_{k}$ is sub-Gaussian on $\h$;
		\item $G$ is of step $\geqslant 3$ and the distribution of each $X_{k}$ is either a standard Gaussian on $\mathcal{H}$ or has bounded support.
	\end{enumerate}
Then for
\[
S_n:=\exp(X_1)\cdots\exp(X_n),
\]
the measures $\{\mu_n\}_{n=1}^\infty$ satisfy a large deviation principle with rate function
\begin{equation}\label{e.1.1}
J(g) : = \inf\left\{\int_0^1 \Lambda^{\ast}(c_{\gamma}\left( t \right))\,dt :  \gamma:[0,1]\to G \text{ horizontal, } \gamma(0)=e, \gamma(1)=g\right\},
\end{equation}
where $c_{\gamma}$ is the Maurer-Cartan form for the horizontal path $\gamma$.
\end{theorem}

\begin{remark}
As usual, there exists a modification of the above statement to accommodate non-centered distributions. But for simplicity we will keep the mean $0$ assumption.
\end{remark}
As is usual for LDPs, Theorem \ref{t.main} characterizes the large deviations rate function as the solution of a variational problem. Since the variational problem in \eqref{e.1.1} is defined on the path space, one can see the connection between the geometry of the Carnot group and the LDPs of the random walks. Keeping this connection in mind, it is interesting to determine whether the variational problem admits explicit solutions that express the underlying metric structure more clearly.

As a particular case, we are able to determine an explicit solution to \eqref{e.1.1} in the important case of Gaussian random walks. When $\{X_n\}_{n=1}^\infty$ are i.i.d.~normally distributed random vectors in $\mathcal{H}$, the rate function \eqref{e.1.1} has the following explicit expression
\begin{align}\label{eq-rate-normal}
 J_{\mathcal{N}}(g) &= \inf\left\{\frac{1}{2}\int_0^1 |\gamma^{\prime}(t)|_{\gamma(t)}^2\,dt: \gamma \text{ horizontal, } \gamma(0)=e, \gamma(1)=g\right\} 
\end{align}
which gives the exact minimum energy to reach $g$ from $e$ at time $1$.
In particular, the rate function in \eqref{eq-rate-normal} can be described in terms of the natural geometry on the group $G$. As we recall in Section~\ref{s.carnot},  H\"ormander's condition implies that any two points in $G$ can be connected by a horizontal path by Chow--Rashevskii's theorem. One may define the (finite) Carnot--Carath\'{e}odory distance $\rho_{cc}(x, y)$ between two points $x, y\in G$ to be the length of the shortest horizontal path connecting $x$ and $y$ (see Section \ref{sss.Carnotmetric} for details). With this definition we can state our results for Gaussian random walks. 
\begin{corollary}\label{c.3.13}
Let $G$ be a homogeneous Carnot group and suppose $\{X_n\}$ are independent $\mathcal{N}(0, \mathrm{Id}_{\mathcal{H}})$ random variables on $\h$. Then Theorem~\ref{t.main} holds for the associated random walk with the rate function
	\begin{align*}
	J_{\mathcal{N}}(x) =\frac12\rho^2_{cc}(e,x),\quad x\in G.
	\end{align*}
\end{corollary}

Since Carnot groups can be identified with copies of $\mathbb{R}^N$ equipped with non-commutative group operations, Corollary~\ref{c.3.13} can be viewed as a broad generalization of the standard LDP for Gaussian random walks in Euclidean spaces, where the rate function is given by $\frac{1}{2}\|x\|^2$. More generally, the rate function in \eqref{eq-rate-normal} is in line with that for LDPs on path space, like the classical one for Brownian motion (that is, Schilder's Theorem) and other diffusions on $\R^n$. Such path space LDPs for continuous-time processes are sometimes accessible from Cram\'{e}r-type results via finite-dimensional distributions, and it's natural to ask if the results of the present paper may be used to prove path space LDPs for the associated hypoelliptic diffusions (\'{a} la Schilder) or polygonal and piecewise-linear paths (\'{a} la Mogulskii). While LDPs are known for some hypoelliptic diffusions (see for example \cite{Azencott1980}) proving path space LDPs from our starting point typically requires additional assumptions, like ellipticity of the generator, or Cameron-Martin-Maruyama-type quasi-invariance, or Girsanov-type results. These conditions are generally not available in our present setting, except  for a limited class of sub-Riemannian manifolds, see \cite{BaudoinFengGordina2019}. 
We also mention that Schilder-type LDPs are known from \cite{Azencott1980} for a large class of hypoelliptic diffusions  with the rate function being the classical one for Brownian motion on $\mathbb{R}^n$.

At this point let us comment a bit about the proof of Theorem \ref{t.main}, and hence also Corollary \ref{c.3.13}. As mentioned above, the case of hypoelliptic diffusions poses a challenge in general, since the setting lacks the appropriate conditions required for most classical LDPs. Thus, our approach will need a new construct, and in particular, be tailored to the sub-Riemannian geometry of group.  Indeed, our approach is different from previous proofs for LDPs: we treat the random walk as an intrinsic stochastic process on the Carnot group viewed as a geodesic metric space. As a result, and in contrast to previous works, we show that by taking horizontal (geodesic) paths connecting $S_n$ and $S_{n+1}$ instead of a piecewise linearization of the continuous-time processes, we recover \emph{geometrically natural} rate functions. This procedure elucidates why the LDPs in this paper relate the random walks to the geometry of the underlying space, as in Theorem \ref{t.main}.

We now compare our results to two particular papers on LDPs for random walks in Lie groups. First consider  \cite{Versendaal2019b}, in which the author studies Cram\'{e}r-type LDPs on Lie groups equipped with a Riemannian structure. While the main results in our paper and in \cite{Versendaal2019b} might appear similar, the LDPs are fundamentally different both in the way the random walks are constructed and the geometric structures of the underlying spaces. In addition, we are able to prove LDPs for a more general class of samples $\left\{ X_n \right\}_{n=1}^{\infty}$ than the bounded distributions considered in \cite{Versendaal2019b}. This allows us to include the very natural case of standard normal sampling, the one and only case that leads to the LDP with the rate function in the energy form given by \eqref{eq-rate-normal}. One of the tools we rely on are concentration inequalities for polynomials of Gaussian random vectors.

Now consider \cite{BaldiCaramellino1999}, in which the authors study LDPs on nilpotent Lie groups. While in the present paper we focus on random walks taking steps only in horizontal directions, we see that the LDPs proved here should be comparable to those for random walks taking steps in arbitrary directions, as considered in \cite{BaldiCaramellino1999}; see Remark \ref{r.comp}. Significantly, the approach of the present paper allows us to show that the infimum is achieved specifically on horizontal paths, clearly tying the rate function to the underlying natural geometry of the space, as we see for example by identifying the rate function for Gaussian samples as described earlier in the introduction.

In \cite{BaldiCaramellino1999} the authors use the fact that these groups can be identified with Euclidean spaces to describe the random walks as an end point process of solutions to stochastic differential equations. As mentioned earlier, LDPs for hypoelliptic diffusions in $\R^n$ that are solutions of SDEs are well-known due to work by Azencott \cite{Azencott1980}. This approach is a natural application of classical LDP results to spaces like nilpotent Lie groups identified with Euclidean space; however, it cannot be used for more general geometric structures. We also note that as a consequence of this approach the group $G$ is equipped with the Euclidean metric structure. There is a significant difference when one looks at the group equipped with a Riemannian or sub-Riemannian distance, as it should be the metric structure that determines the rate functions for the LDP, as we show in the case of homogeneous Carnot groups. Working instead in the tangent space as in the present paper offers the potential to generalize to other sub-Riemannian manifolds. 

Our paper is organized as follows. In Section~\ref{s.background} we provide the terminology needed to state and prove Theorem~\ref{t.main} precisely, and in Section~\ref{s.LDP} we prove Theorem~\ref{t.main}. Note that it is only the arguments of Section~\ref{ss.approx}, where we show we have exponentially good approximations of the walk, that require the additional assumptions (sub-Gaussian or Gaussian or bounded) on the distribution.

\section{Background and setup}
\label{s.background}

\subsection{Large deviations principles on metric spaces}

We first recall some basic definitions for large deviations principles. These can be found for example in \cite[Section 1.2]{DemboZeitouniBook2010}. Suppose $\mathcal{M}$ is a Hausdorff topological space.

\begin{definition} A function $I: \mathcal{M} \longrightarrow [0, \infty]$ is called a \emph{rate function} if $I$ is not  identically $\infty$ and if $I$ is lower semi-continuous. That is, the set $\left\{x \in \mathcal{M}: I\left( x \right) \leqslant a \right\}$ is closed for every $a \geqslant 0$. $I$ is called a \emph{good rate function} if in addition $\left\{x \in \mathcal{M}: I\left( x \right) \leqslant a \right\}$ is compact for every $a \geqslant 0$.
\end{definition}
In our setting $\mathcal{M}=\left(\mathcal{M}, \rho \right)$ is a metric space, and therefore one can verify lower semi-continuity on sequences. I.e., $I$ is lower semi-continuous if and only if
\[
\liminf_{x_{n} \to x} I\left( x_{n} \right) \geqslant I\left( x \right)
\]
for all $x \in \mathcal{M}$. This means that a good rate function on $\mathcal{M}$ is a rate function that achieves its infimum over closed sets. We denote by $\mathcal{B}$ the (complete)  Borel $\sigma$-algebra over the metric space $\mathcal{M}$.

\begin{definition} A sequence of probability measures $\left\{ \mu_{n} \right\}_{n=1}^{\infty}$ on $\left( \mathcal{M}, \mathcal{B} \right)$  satisfies the \emph{large deviation principle} (LDP) with the rate function $I:\mathcal{M} \longrightarrow [0, \infty]$ if the function on $\mathcal{B}$ defined as $\frac{1}{n} \log \mu_{n}\left(  B \right)$, $B \in \mathcal{B}$,  converges weakly to the function $-\inf_{x \in B} I\left( x \right)$. Equivalently, for any open set $O\subset \mathcal{M}$ and any closed set $F \subset \mathcal{M}$
\begin{align}
& \liminf_{n \to \infty}\frac{1}{n} \log \mu_{n}\left(  O \right) \geqslant - \inf_{x\in O} I(x),
\notag
\\
& \limsup_{n \to \infty}\frac{1}{n} \log \mu_{n}\left( F \right) \leqslant - \inf_{x\in F} I(x).\label{e.full}
\end{align}
We say that $\left\{ \mu_{n} \right\}_{n=1}^{\infty}$ satisfies the \emph{weak LDP} if \eqref{e.full} holds for all compact sets $F\subset\mathcal{M}$.

Similarly, for a sequence of $\mathcal{M}$-valued random variables $\{Z_n\}_{n=1}^\infty$ defined on probability spaces $(\Omega,\mathcal{B}_n,P_n)$, we say $\{Z_n\}_{n=1}^\infty$ satisfies \emph{the (weak) LDP} with the rate function $I$ if the sequence of push-forward measures $\{P_n\circ Z_n^{-1}\}_{n=1}^\infty$ satisfies the (weak) LDP with the rate function $I$.
\end{definition}
We will also need the following standard fact known as the contraction principle (see for example \cite[Theorem 4.2.1]{DemboZeitouniBook2010}).
\begin{theorem}[Contraction principle]\label{thm-contract}
Suppose $\mathcal{M}$ and $\mathcal{N}$ are Hausdorff topological spaces and $f:\mathcal{M}\to\mathcal{N}$ is a continuous map. Let $I:\mathcal{M}\to[0,\infty]$ be a good rate function.
\begin{itemize}
\item[(i)] For any $y\in \mathcal{N}$, define
\[
J(y):=\inf\{I(x): x\in\mathcal{M} \text{ and } y=f(x) \},
\]
then $J:\mathcal{N}\to[0,\infty]$ is a good rate function on $\mathcal{N}$, where as usual the infimum over empty set is taken as $\infty$.
\item[(ii)] If $I$ controls the LDP associated with a family of probability measures $\{\mu_n\}_{n\ge1}$ on $\mathcal{M}$, then $J$ controls the LDP associated with the push-forwards $\{\mu_n\circ f^{-1}\}_{n\ge1}$ on $\mathcal{N}$.
\end{itemize}
\end{theorem}
Also important in the theory of LDP are the notions of exponential equivalence and exponential approximation. For the following definition, see for example \cite[Definition 4.2.10]{DemboZeitouniBook2010}.
\begin{definition}\label{d.exp-equiv}
For $n\in \mathbb{Z}^{+}$, let $\left( \Omega,\mathcal{B}_n, P_{n} \right)$ be probability spaces and $Z_n$ and $\widetilde{Z}_{n}$ be sequences of $\mathcal{M}$-valued random variables with joint laws $P_n$. Then $\{Z_{n}\}$ and $\{\widetilde{Z}_n\}$ are called \emph{exponentially equivalent} if for every $\delta > 0$ and $\Gamma_\delta:= \{(x,y):\rho(x,y)>\delta\}\subset\mathcal{M}\times\mathcal{M}$, the set $\{(\widetilde{Z}_{n}, Z_{n})\in\Gamma_\delta\}\in\mathcal{B}_n$ and
\[
\limsup_{n\to\infty}\frac{1}{n} \log P_{n}(\Gamma_\delta) = -\infty.
\]
\end{definition}

The following theorem records the known relationship between the LDPs of exponentially equivalent families of random variables; see for example \cite[Theorem 4.2.13]{DemboZeitouniBook2010}.

\begin{theorem}[Theorem~4.2.13 in \cite{DemboZeitouniBook2010}]\label{t.exp-equiv} Suppose that $\{Z_{n}\}$ is exponentially equivalent to $\{\widetilde{Z}_n\}$ and satisfies an LDP with a good rate function. Then $\{\widetilde{Z}_{n}\}$ also satisfies an LDP with the same rate function.
\end{theorem}

For the following see for example \cite[Definition 4.2.14]{DemboZeitouniBook2010}.
\begin{definition}\label{d.exp-approx}
	For $n, m \in \mathbb{Z}^{+}$, let $\left( \Omega,\mathcal{B}_n, P_{n,m} \right)$ be probability spaces, and $\widetilde{Z}_n$ and $Z_{n,m}$ be sequences of $\mathcal{M}$-valued random variables with joint laws $P_{n,m}$. Then $\{Z_{n,m}\}$ are called \emph{exponentially good approximations} of $\{\widetilde{Z}_n\}$ if for every $\delta > 0$ and $\Gamma_\delta:= \{(x,y):\rho(x,y)>\delta\}\subset\mathcal{M}\times\mathcal{M}$, the set $\{(\widetilde{Z}_{n}, Z_{n, m})\in\Gamma_\delta\}\in\mathcal{B}_n$ and
\[
\lim_{m\to\infty} \limsup_{n\to\infty}\frac{1}{n} \log P_{n,m}(\Gamma_\delta) = -\infty.
\]
\end{definition}

The following theorem records the known relationship between the LDPs of families of exponentially good approximations; see for example \cite[Theorem 4.2.16]{DemboZeitouniBook2010}.

\begin{theorem}[Theorem~4.2.16 in \cite{DemboZeitouniBook2010}]\label{t.exp-approx} Suppose that for every $m$, the family of random variables $\{Z_{n,m}\}_{n=1}^\infty$ satisfies the LDP with the rate function $I_m$ and that $\{Z_{n,m}\}$ are exponentially good approximations of $\{\widetilde{Z}_n\}$.

\begin{enumerate}[(i)]
		\item Then $\{\widetilde{Z}_n\}_{n=1}^\infty$ satisfies a weak LDP with the rate function
			\[ I(y) := \sup_{\delta>0}\liminf_{m\to\infty}\inf_{z\in B_{y,\delta}} I_m(z) \]
			where $B_{y,\delta}$ denotes the ball $\{z:\rho(y,z)<\delta\}$.
		\item If $I$ is a good rate function and for every closed set $F$
			\[ \inf_{y\in F} I(y) \leqslant \limsup_{m\to\infty}\inf_{y\in F} I_m(y) \]
			then the full LDP holds for $\{\widetilde{Z}_n\}_{n=1}^\infty$ with the rate function $I$.
	\end{enumerate}
\end{theorem}

\subsection{Carnot groups}\label{s.carnot} In this paper we concentrate on a particular class of metric spaces, namely, homogeneous Carnot groups equipped with the Carnot-Carath\'{e}odory metric (by \cite[Proposition 2.2.17, Proposition 2.2.18]{BonfiglioliLanconelliUguzzoniBook}, the assumption about homogeneity is without loss of generality). We begin by recalling  basic facts about Carnot (stratified) groups that we require in the sequel. For the uninitiated reader we have tried to be as comprehensive as possible in our exposition. Any missing details as well as further elaboration can be found in a number of references, see for example \cite{BonfiglioliLanconelliUguzzoniBook, VaropoulosBook1992}.

\subsubsection{Carnot groups as Lie groups} We say that $G$ is a Carnot group of step $r$ if $G$  is a connected and simply connected Lie group whose Lie algebra $\mathfrak{g}$ is \emph{stratified}, that is, it can be written as
\[
\mathfrak{g}=V_{1}\oplus\cdots\oplus V_{r},
\]
where
\begin{align}\label{e.Stratification}
& \left[V_{1}, V_{i-1}\right]=V_{i}, \hskip0.1in 2 \leqslant i \leqslant r,
\notag
\\
& [ V_{1}, V_{r} ]=\left\{ 0 \right\}.
\end{align}
To exclude trivial cases we assume that the dimension of $\mathfrak{g}$ is at least $3$. In addition we will use a stratification such that the center of $\mathfrak{g}$ is contained in  $V_{r}$. In particular, Carnot groups are nilpotent. For $\mathcal{X}\in\mathfrak{g}$ we will write
\begin{equation*}
 \mathcal{X} = \mathcal{X}^{(1)}+\cdots+\mathcal{X}^{(r)} \in V_1\oplus\cdots\oplus V_r. 
 \end{equation*}

\begin{notation}\label{n.startification} By $\mathcal{H}:=V_{1}$ we denote the space of \emph{horizontal} vectors that generate the rest of the Lie algebra with $V_{2}=[\h,\h], \ldots, V_r = \h^{(r)}$.
\end{notation}
As usual, we let
\begin{align*}
\exp&: \mathfrak{g} \longrightarrow G,
\\
\log&: G \longrightarrow \mathfrak{g}
\end{align*}
denote the exponential and logarithmic maps, which  are global diffeomorphisms for connected nilpotent groups, see for example \cite[Theorem 1.2.1]{CorwinGreenleafBook}. Also, for $\X\in\mathfrak{g}$, we let $\operatorname{ad}_\X:\mathfrak{g}\to\mathfrak{g}$ denote the \emph{adjoint map} defined by $\operatorname{ad}_\X\Y:=[\X,\Y], \Y \in \mathfrak{g}$.

\subsubsection{Identifying $G$ with a linear space}
Since $\exp$ and $\log$ are global diffeomorphism between $G$ and $\mathfrak{g}$, we obtain a natural way to identify $G$ with a linear space, underlying its Lie algebra and equipped with some non-trivial group law. We now explain how this identification works. First, by identifying $\mathfrak{g}$ with $\mathbb{R}^{N}$ we can obtain the following notion of stratified coordinates.

\begin{definition} A set $\left\{ \X_{1}, \ldots , \X_{N} \right\}\subset\mathfrak{g}$  is a \emph{basis for $\mathfrak{g}$ adapted to the stratification} if the subset $\left\{ \X_{ d_0+d_1+\cdots+d_{i-1}+j} \right\}_{j=1}^{d_i}$ is a basis of $V_{i}$ for each $i\in [r]$,  where we adopt the standard notation $[r] := \{1,\ldots, r\}$ for $r \in \N$.
\end{definition}

We now recall the Baker-Campbell-Dynkin-Hausdorff formula, expressing the group product in terms of the Lie algebra. Since $\mathfrak{g}$ is nilpotent, the formula takes a particular appealing form and allows to present the group multiplication by polynomials. For the following version, see for example \cite[p.~585, Equation(4.12)]{BonfiglioliLanconelliUguzzoniBook} or \cite[p.~11]{CorwinGreenleafBook}.
\begin{notation}\label{n.BCDH}
For any $\X, \Y \in \mathfrak{g}$, the \emph{Baker-Campbell-Dynkin-Hausdorff formula} is given by

\begin{equation}\label{e.BCDH}\begin{split}
BCDH(\X,\Y)&:= \log(e^\X e^\Y)
\\
& = \X+\Y+\sum_{k=1}^{r-1} \sum_{(n,m)\in\mathcal{I}_{k}}
		a_{n, m}^k\operatorname{ad}_\X^{n_1} \operatorname{ad}_\Y^{m_1} \cdots
		\operatorname{ad}_\X^{n_{k}} \operatorname{ad}_\Y^{m_{k}} \X,
\end{split}\end{equation}
where
\begin{equation}
\label{e.ak}
a_{n,m}^k := \frac{(-1)^k}{(k+1)m!n!(|n|+1)},
\end{equation}
$\mathcal{I}_{k} := \{(n,m)\in\mathbb{Z}_+^k\times\mathbb{Z}_+^k :
n_i+m_i>0 \text{ for all } i \in [k] \}$, and for each multi-index
$n\in\mathbb{Z}_+^k$,
\[ n!= n_1!\cdots n_{k}! \quad \text{ and } \quad |n|=n_1+\cdots+n_{k}. \]
\end{notation}

Since $\mathfrak{g}$ is nilpotent of step $r$ we have
\[
\operatorname{ad}_\X^{n_1} \operatorname{ad}_\Y^{m_1} \cdots
		\operatorname{ad}_\X^{n_{k}} \operatorname{ad}_\Y^{m_{k}} \X = 0 \quad
\text{if } |n|+|m|\geqslant r
\]
for $\X,\Y\in\mathfrak{g}$.  

The Baker-Campbell-Dynkin-Hausdorff formula suggests it could be beneficial to lift coordinates from the Lie algebra $\mathfrak{g}$ to the group $G$.

\begin{definition}\label{d.ExpCoord} A system of \emph{exponential coordinates (of the first kind)}, relative to a basis $\left\{ \X_{1}, \dots, \X_{N} \right\}$ of $\mathfrak{g}$ adapted to the stratification, is a map from $\mathbb{R}^{N}$ to $G$ defined by
\[
x \longmapsto \exp\left( \sum_{i=1}^{N} x_{i}\X_{i} \right), \text{ where } x=\left( x_{1}, \dots, x_{N} \right) \in \mathbb{R}^{N}.
\]
With the exponential coordinates we can now equip $\mathbb{R}^{N}$ with a group operation pulled back from $G$ by
\begin{align*}
z&:=x \star y,
\\
	\sum_{i=1}^{N} z_{i}\X_{i}&=\operatorname{BCDH}\left( \sum_{i=1}^{N} x_{i}\X_{i}, \sum_{i=1}^{N} y_{j}\X_{j} \right).
\end{align*}
\end{definition}
In particular, in this identification $x^{-1}=-x$. Note that $\mathbb{R}^{N}$ with this group law is a Lie group whose Lie algebra is isomorphic to $\mathfrak{g}$. Both $G$ and $\left( \mathbb{R}^{N}, \star \right)$ are nilpotent, connected and simply connected, therefore the exponential coordinates give a diffeomorphism between $G$ and $\mathbb{R}^{N}$. Thus we identify both $G$ and $\mathfrak{g}$ with $\mathbb{R}^{N}$. For $x=\exp(\mathcal{X})\in G$ with $\mathcal{X}=\sum_{i=1}^Nx_i\mathcal{X}_i$, we will write
\begin{equation}\label{eq-step-cord}
x=(x^{(1)},\ldots,x^{(r)}) \in \R^{d_1}\times\cdots\times\R^{d_r},
\end{equation}
where $x^{(j)}=(x_{d_1+\cdots+d_{j-1}+1},\ldots,x_{d_1+\cdots+d_{j}})$. 
We shall henceforth always identify a homogeneous Carnot group $G$ with the $(\R^N, \star)$.

\textbf{Dilations:}  Analogous to scaling in Euclidean normed spaces, a stratified Lie algebra is equipped with a natural family of \emph{dilations} defined for any $a>0$ by
\[
	\delta_{a}\left( \X \right):=a^i \X,  \text{ for } \X \in V_{i}.
\]
For each $a>0$, $\delta_a$ is a Lie algebra isomorphism, and the family of all dilations $\{\delta_a\}_{a>0}$ forms a one-parameter group of Lie algebra isomorphisms. We use the identification between $G$ and $\mathfrak{g}$ to define similar automorphisms $D_{a}$ on $G$. The maps $D_{a}:=\exp \circ \, \delta_{a} \circ \log: G \longrightarrow G$ satisfy the following properties.
\begin{equation} \label{e.Dilations}
	\begin{split}
&  D_{a}\circ \exp=\exp \circ\, \delta_{a} \quad\text{ for any } a >0,
\\
&  D_{a_{1}}\circ D_{a_{2}}=D_{a_{1}a_{2}}, D_{1}=I \quad\text{ for any } a_{1}, a_{2}>0,
\\
&  D_{a}\left( g_{1} \right)D_{a}\left( g_{2} \right)=D_{a}\left( g_{1}g_{2} \right) \quad\text{ for any } a>0 \text{ and } g_{1}, g_{2} \in G,
\end{split}\end{equation}
That is, the group $G$ has a family of dilations which is adapted to its stratified structure. Actually, $D_{a}$ is the unique Lie group automorphism corresponding to $\delta_a$ in the sense that $dD_{a}=\delta_{a}$.
On a homogeneous Carnot group $\mathbb{R}^{N}$ the dilation $D_{a}$ can be described explicitly by
\begin{equation*}
D_{a}\left( x_{1}, \dots, x_{N} \right):=\left( a^{\sigma_{1}}x_{1}, \dots, a^{\sigma_{N}}x_{N}\right),
\end{equation*}
	where $\sigma_{j} \in\{1,\ldots,r\}$ is called the \emph{homogeneity} of $x_{j}$, with
\[
\sigma_{j}= i, \qquad\text{ for } d_0+d_1+\cdots+d_{i-1}< j \leqslant d_1+\cdots+d_{i},
\]
with $i=1,\ldots,r$ and recalling that $d_0=0$.
That is, $\sigma_{1}=\dots= \sigma_{d_{1}}=1, \sigma_{d_{1}+1}=\cdots=\sigma_{d_1+d_{2}}=2$, and so on. In other words, $\sigma_j = i$, if and only if $x_j  \in V_i$.
\\
\paragraph{\bf Group operations as polynomials:} A key observation for our analysis is that the group operation of a homogeneous Carnot group  $G=\left( \mathbb{R}^{N}, \star \right)$ can be expressed component-wise as homogeneous polynomials. For example, by \cite[Proposition 2.1]{FranchiSerapioniSerra_Cassano2003a}  we have
\begin{align}\label{e.3.3}
x \star y = x + y + Q\left( x, y \right) \text{ for } x, y \in \mathbb{R}^{N},
\end{align}
	where $Q =(Q^{(1)},\ldots,Q^{(r)}): \mathbb{R}^{N}\times \mathbb{R}^{N} \longrightarrow \mathbb{R}^{N}$ with $Q^{(i)}=\left( Q_{d_1+\cdots+d_{i-1}+1}, \dots, Q_{d_1+\cdots+d_i} \right):\mathbb{R}^{N}\times \mathbb{R}^{N} \longrightarrow \R^{d_i}$ and each $Q_{j}$ is a homogeneous polynomial of degree $\sigma_{j}$ with respect to the dilations $D_{a}$ of $G$,
\[
Q_{j}\left( D_{a} x,  D_{a} y \right)=a^{\sigma_{j}} Q_{j}\left( x,   y \right), \qquad \text{ for } x, y \in \R^N.
\]
Moreover, for any $x, y \in G$ we have
\begin{align*}
& Q_{1}\left( x, y \right)=\dots=Q_{d_{1}}\left( x, y \right)=0,
\\
& Q_{j}\left( x, 0 \right)=Q_{j}\left( 0, y \right)=0, Q_{j}\left( x, x \right)=Q_{j}\left( x, -x \right)=0, \quad \text{ for } d_{1}< j \leqslant N,
\end{align*}
and for $d_{1}+\dots+d_{i}< j \leqslant d_{1}+\dots+d_{i+1}$
\begin{align*}
	Q_{j}\left( x, y \right)=Q_{j}\left( (x^{(1)},\ldots,x^{(i)}), (y^{(1)},\ldots,y^{(i)})\right)
	=-Q_{j}\left( -y, -x \right).
\end{align*}
In particular, \eqref{e.3.3} gives a direct argument to see that group operations on homogeneous Carnot groups are differentiable. Note that \cite{FranchiSerapioniSerra_Cassano2003a} uses a slightly different notation $h_{i}:= d_{1}+\dots+d_{i}$.

In addition by \cite[Proposition 2.2.22 (4)]{BonfiglioliLanconelliUguzzoniBook} we have for $j=d_{1}+1, \dots, N$
\begin{equation}\label{e.SymplForm}
	Q_{j}\left( x, y \right)=\sum_{(k, \ell)\in I_j} \left( x_{k}y_{\ell}-x_{\ell}y_{k} \right) R_{j}^{k, \ell}\left( x, y \right),
\end{equation}
	where $I_j:=\{(k,\ell):k<\ell, \sigma_{k}+\sigma_{\ell} \leqslant \sigma_{j}\}$ (as described in  \cite[p.\,1951]{FranchiSerapioni2016}) and $R_{j}^{k, \ell}$ are homogeneous polynomials of degree $\sigma_{j}-\sigma_{k}-\sigma_{\ell}$ with respect to the group dilations.

	\begin{notation}[Symplectic form] For any $m$ and vectors $x=\left( x_{1}, \ldots , x_{m} \right), y=\left( y_{1}, \ldots , y_{m} \right)\in\mathbb{R}^m$ define
\[
\omega_{k, \ell}\left( x, y \right):=x_{k}y_{\ell}-x_{\ell}y_{k},
\]
for $1 \leqslant k, \ell \leqslant m$.
\end{notation}
Note that $\omega_{k, \ell}\left( x, y \right)=-\omega_{\ell, k}\left( x, y \right)=-\omega_{k, \ell}\left( -y, -x \right)$, and using this notation we can write \eqref{e.SymplForm} as
\begin{equation}\label{e.2.7}
	Q_{j}\left( x, y \right)=\sum_{(k, \ell)\in I_j} \omega_{k, \ell}\left( x, y \right) R_{j}^{k, \ell}\left( x, y \right),
\end{equation}
for $j=d_{1}+1, \dots, N$, where $R_{j}^{k, \ell}\left( x, y \right)= R_{j}^{k, \ell}\left( -y, -x\right)$.

\begin{example}[Step $2$]\label{ex.step2}
Suppose $G\cong \mathbb{R}^{d_{1}+d_{2}}$ is a homogeneous Carnot group of step $2$. That is, $\mathfrak{g}=\mathcal{H}\oplus \mathcal{V}$ with $\operatorname{dim}\mathcal{H}=d_1$, $\operatorname{dim}\mathcal{V}=d_{2}$, and homogeneity $\sigma_{j}=1$ if $1 \leqslant j \leqslant d_1$ and $\sigma_{j}=2$ if $d_1+1 \leqslant j \leqslant d_1+d_{2}$. Then
\begin{align*}
x \star y &= x + y + \left(
	\mathbf{0}_{d_1},  Q^{(2)}\left( x, y \right)\right) \\
	&= x + y + \left( 
	\mathbf{0}_{d_1},  Q_{d_1+1}\left( x, y \right), \dots,  Q_{d_1+d_{2}}\left( x, y \right)\right).
\end{align*}
Note that, for $j=d_1+1, \dots, d_1+d_{2}$, \eqref{e.SymplForm} becomes
\[
Q_{j}\left( x, y \right)=\sum_{1\leqslant k< \ell\leqslant d_1}
\omega_{k, \ell}\left( x, y \right) R_{j}^{k, \ell}\left( x, y \right),
\]
where the polynomials  $R_{j}^{k, \ell}$ are necessarily constant, since in this case $r=2$ and $\sigma_{k}=\sigma_{\ell}=1$, and therefore $\sigma_{j}-\sigma_{k}-\sigma_{\ell}=0$ for $j=d_1+1, \dots, d_1+d_{2}$.
Thus the components  $Q_{j}$ are skew-symmetric bilinear forms on $\mathbb{R}^{d_1}$, 
\begin{equation}\label{e.2.6-1}
Q_{j}\left( x, y \right)=\sum_{1\leqslant k< \ell\leqslant d_1}
	\alpha_{j}^{k, \ell}\omega_{k, \ell}\left( x^{(1)}, y^{(1)} \right)
\end{equation}
for some constants $\alpha_{j}^{k, \ell}$ and $d_1+1 \leqslant j \leqslant d_1+d_{2}$.

To make it more transparent, we can write $x_{1}=\left( h_{1}, v_{1} \right)$, $x_{2}=\left( h_{2}, v_{2} \right)$, where $h_{1}, h_{2} \in \mathbb{R}^{d_{1}}$ and $v_{1}, v_{2} \in \mathbb{R}^{d_{2}}$, so that
each $Q_{j}$ can be expressed in a matrix form as 
\begin{align*}
	Q_{j}\left( x_{1}, x_{2} \right)= Q_j((h_1,v_1),(h_{2},v_{2}))
	 =h_{1}^{T}A_{j} h_{2},
\end{align*}
with
\[ A_{j}:=\left( \alpha_{j}^{k, \ell} \right)_{k,\ell=1}^{d_{1}}, \quad \alpha_{j}^{k, \ell}=-\alpha_{j}^{\ell, k}.
	\]

We can use this representation of the group multiplication to write the product of $n$ elements $x_{1}, \ldots , x_{n} \in G$, where $x_{i}=\left( h_{i}, v_{i} \right)$ with $h_{i} \in \mathbb{R}^{d_{1}}$ and $v_{i}=\left( v_{i}^{d_{1}+1}, \ldots, v_{i}^{d_{1}+d_{2}}\right) \in \mathbb{R}^{d_{2}}$,  as follows.
\begin{align*}
	\prod_{i=1}^{n}& \star\, x_{i}
	:= x_1\star\cdots\star x_n \\
	&=\left( \sum_{i=1}^{n} h_{i},
		\sum_{i=1}^{n}v_{i}^{d_{1}+1}+\sum_{1\leqslant k< \ell\leqslant d_1} h_{k}^{T}A_{d_{1}+1} h_{\ell}, \ldots, \sum_{i=1}^{n} v_{i}^{d_{1}+d_{2}}+\sum_{1\leqslant k< \ell\leqslant d_1} h_{k}^{T}A_{d_{1}+d_{2}} h_{\ell} \right).
\end{align*}
	In particular, when $v_i=\mathbf{0}_{d_{2}}$ for all $i$
\begin{equation}\label{e.2.5}
	\prod_{i=1}^{n}\star \left( h_{i}, \mathbf{0}_{d_{2}} \right)=
\left( \sum_{i=1}^{n} h_{i},
\sum_{1\leqslant k< \ell\leqslant d_1} h_{k}^{T}A_{d_{1}+1} h_{\ell}, \ldots , \sum_{1\leqslant k< \ell\leqslant d_1} h_{k}^{T}A_{d_{1}+d_{2}} h_{\ell} \right),
\end{equation}
and in the case $d_{2}=1$ we have
\begin{equation*}
\prod_{i=1}^{n} \star \left( h_{i}, 0 \right)=
\left( \sum_{i=1}^{n} h_{i},
\sum_{1\leqslant k< \ell\leqslant d_1} h_{k}^{T}A h_{\ell}\right)
\end{equation*}
for a skew-symmetric matrix $A$.
\end{example}

\begin{example}[Heisenberg groups]\label{ex.Heisenberg}
The (Isotropic) Heisenberg group is an example of a step $2$ group with $d_{2}=1$, $d_{1}=2d$ and $A$ being a $2d \times 2d$ matrix with the blocks

\[
\left(
  \begin{array}{cc}
    0 & -1 \\
    1 & 0 \\
  \end{array}
\right)
\]
on the diagonal. Then for $\left( x_{i}, y_{i}, z_{i}\right) \in G \cong \mathbb{R}^{2n}\times \mathbb{R}, i \in[n]$, we have

\[
\prod_{i=1}^{n} \star \left( x_{i}, y_{i}, z_{i}\right)=\left( \sum_{i=1}^{n} x_{i}, \sum_{i=1}^{n} y_{i}, \sum_{i=1}^{n} z_{i}+\sum_{1\leqslant i< j\leqslant n} \left( x_{i}, y_{i} \right) A \left( x_{j}, y_{j} \right)^{T}\right).
\]
\end{example}

\begin{example}[Step $3$ Engel group] This is a group of step $3$ that can be modeled on $\mathbb{R}^{4}$, with $\mathcal{H}=\mathbb{R}^{2} \times \left\{ 0 \right\}$.  The multiplication is given by
\begin{align*}
& x \star y= x+y+\left( 0, 0, \frac{1}{2}\omega_{1,2}\left( x, y \right), \frac{1}{2}\omega_{1, 3}\left( x, y \right)+\frac{1}{12} \omega_{1, 2}\left( x, y\right)\left( x_{1}-y_{1} \right) \right),
\end{align*}
so that
	\begin{align*}
& \left( x_{1}, x_{2}, 0, 0 \right) \star \left( y_{1}, y_{2}, 0, 0 \right) = \left( x_{1}+y_{1}, x_{2}+y_{2}, \frac{1}{2}\omega_{1,2}\left( x, y \right), \frac{1}{12} \omega_{1, 2}\left( x, y\right)\left( x_{1}-y_{1} \right) \right)
\end{align*}
and
\begin{align*}
& \left( x_{1}, x_{2}, x_{3}, x_{4} \right) \star \left( -x_{1}, -x_{2}, 0, 0 \right)= \left( 0, 0, x_{3}, x_{4}+\frac{1}{2}x_{1}x_{3} \right).
\end{align*}
\end{example}

The following observation will be useful in the sequel and is a straightforward consequence of the Baker-Campbell-Hausdorff-Dynkin formula.

\begin{lemma}\label{l.Il} Suppose $G$ is a homogeneous Carnot group of step $r$ and $\X_1,\ldots,\X_{k}$ are any elements of $\h$. Then for $\ell=2,\ldots,r$ 	\[ (e^{\X_1}\star\cdots\star e^{\X_n})^{(\ell)}
		= \sum_{i=(i_1,\ldots,i_\ell)\in \mathcal{J}_\ell} c_i \operatorname{ad}_{\X_{i_1}}\cdots \operatorname{ad}_{\X_{i_{\ell-1}}}
			\X_{i_\ell} \]
	for some coefficients $|c_i|<1$, where $\mathcal{J}_\ell$ is some strict subset of $\{1,\ldots,n\}^\ell$ and thus $\#\mathcal{J}_\ell\leqslant n^\ell$.
\end{lemma}

Finally, we can also express some differentials of the group law as polynomials over the Lie algebra. Since we concern ourselves with differentials, we recall that the Lie algebra $\mathfrak{g}$ is naturally identified with the tangent space $T_xG$, for any $x \in G$.
Thus, we will still continue to use $\X$ to denote elements of the Lie algebra which are being mapped to elements of the group $x=e^\X$, but we will also begin at this point to use more traditional vector notation $u,v$, etc.\,to denote elements of the Lie algebra and more generally in the tangent spaces of $G$.

We now introduce the multiplication operator and its differential.
\begin{notation}\label{n.2.11} For $x\in G$, we denote by $L_{x}: G \longrightarrow G$ the \emph{left multiplication}
\[
L_{x} y:=x \star y,  \text{ for } y \in G,
\]
and the corresponding pushforward (differential) $(L_x)_{\ast}:TG \to TG$ by
\begin{align*}
\left( L_{x} \right)_{\ast}: T_{y}G &\rightarrow T_{xy}G
\\
v &\mapsto (L_x)_{\ast}v.
\end{align*}
The \emph{Maurer-Cartan form} $\omega$ is the $\mathfrak{g}$-valued one-form on $G$ defined by
\begin{align*}
& \omega: T_{x}G \longrightarrow T_{e}G \cong \mathfrak{g},\quad v \in  T_{x}G,
\\
& \omega\left( v \right):=(L_{x^{-1}})_{\ast}v \in \mathfrak{g}.
\end{align*}
\end{notation}

The next statement describes the pushforward of left multiplication on elements of the Lie algebra. It is a corollary of \cite[Proposition 3.15]{Melcher2009a}, but we include a proof here for convenience.
\begin{prop}\label{p.tt} Let $\mathcal{X}\in \mathfrak{g}$, $x:=e^\X\in G$, and $v\in T_eG =  \mathfrak{g}$. Then
	\begin{equation*}
		(L_x)_{\ast}v = \sum_{n=0}^{r-1}  A_n \operatorname{ad}_\X^{n} v
	\end{equation*}
	where $A_0=1$ and for $n=1,\ldots,r-1$
	\[ A_n = -\sum_{k=1}^{r-1}\sum_{(n,m)} a_{n,m}^k \]
	where the second sum is over $(n,m)\in\mathcal{I}_{k}$ so that $m_1=\cdots=m_{k-1}=0$, $m_{k}=1$, and $|n|+|m|<r$ and $a_{n,m}^k$ are as in \eqref{e.ak}. Equivalently,
	there exist polynomials $C_{j}^{k,\ell}$ in $x$ so that
	\[(L_x)_{\ast}v = v + \left( \mathbf{0}_{d_1}, \sum_{(k,\ell)\in {I}_{d_1+1}} C_{d_1+1}^{k,\ell}(x)\omega_{k,\ell}(x,v),\ldots, \sum_{(k,\ell)\in {I}_N} C_N^{k,\ell}(x)\omega_{k,\ell}(x,v) \right) \]
	where again $I_j:=\{(k,\ell):k<\ell, \sigma_{k}+\sigma_{\ell} \leqslant \sigma_{j}\}$.
In particular, when $G$ is step 2, these polynomials are constants and
	\[(L_x)_{\ast}v = v + \left( \mathbf{0}_{d_1}, \sum_{1\leqslant k<\ell\leqslant d_1} \alpha_{d_1+1}^{k,\ell}\omega_{k,\ell}(x^{(1)},v^{(1)}),\ldots, \sum_{1\leqslant k<\ell\leqslant d_1} \alpha_N^{k,\ell}\omega_{k,\ell}(x^{(1)},v^{(1)}) \right). \]
\end{prop}
\begin{proof}
Let $\gamma(t):=e^{tv}$, so that $\displaystyle(L_x)_{\ast}v =\frac{d}{dt}\bigg|_0 x\star \gamma(t)$.
	Using \eqref{e.BCDH} we may write
\[		
x\star \gamma(t)
=\left(\X+ tv + \sum_{k=1}^{r-1}\sum_{(n,m)\in\mathcal{I}_{k}}
			a_{n,m}^k \operatorname{ad}_\X^{n_1} \operatorname{ad}_{tv}^{m_1} \cdots
		\operatorname{ad}_\X^{n_{k}} \operatorname{ad}_{tv}^{m_{k}} \X\right). 
\]
Then the first expression follows from noting that, for each term in the sum,
\begin{align*}
\frac{d}{dt}\bigg|_0 \operatorname{ad}_\X^{n_1} & \operatorname{ad}_{tv}^{m_1} \cdots
		\operatorname{ad}_\X^{n_{k}} \operatorname{ad}_{tv}^{m_{k}} \X
\\
&=\left\{\begin{array}{cl} \operatorname{ad}_\X^{|n|}\operatorname{ad}_v \X & \text{if } m_{k}=1 \text{ and }
			m_1=\cdots=m_{k-1}=0\\
			0 & \text{otherwise }
		\end{array}\right. .
\end{align*}

	Alternatively, following \eqref{e.3.3} we may write
\begin{align*}
\left.\frac{d}{dt}\right|_{0}x \star \gamma\left( t \right)
	&=\left.\frac{d}{dt}\right|_{0} \left(x + \gamma\left( t \right) + Q\left( x, \gamma\left( t \right) \right)\right)
\\
& = v+\left.\frac{d}{dt}\right|_{0}Q\left( x,\gamma\left( t \right) \right),
\end{align*}
and using \eqref{e.2.7} note that the only non-zero terms in the second summand are
\[
\left.\frac{d}{dt}\right|_{0}Q_{j}\left( x, \gamma\left( t \right) \right)
	=\sum_{(k, \ell)\in I_j}\omega_{k,\ell}\left( x, v \right)  R_{j}^{k, \ell}\left( x, 0 \right),
\]
where $R_{j}^{k, \ell}\left( x, 0 \right)$ is a polynomial in $x$ depending only on the structure of $G$. The step 2 case follows from \eqref{e.2.6-1}.
\end{proof}

\subsubsection{Carnot groups as metric spaces}\label{sss.Carnotmetric}
As discussed in the introduction, the inner product $\langle\cdot,\cdot\rangle_\h$ induces a natural sub-Riemannian structure on $G$. We identify the horizontal space $\mathcal{H} \subset \mathfrak{g}$ with $\mathcal{D}_{e}\subset T_{e}G$, and then define $\mathcal{D}_{x}:=(L_x)_{\ast}\mathcal{D}_{e}$ for any $x \in G$. Vectors in $\mathcal{D}$ are called \emph{horizontal}. Recall that we introduced the $\mathfrak{g}$-valued Maurer-Cartan form in Notation~\ref{n.2.11}.

\begin{definition}\label{d.HorizontalPath}
A path $\gamma:[0,1]\to G$ is said to be \emph{horizontal} if $\gamma$ is absolutely continuous and $\gamma^{\prime}(t)\in D_{\gamma(t)}$ for a.e.~$t$, that is, the tangent vector to $\gamma\left(t\right)$ at a.e. point of $\gamma \left(t\right)$ is horizontal. Equivalently, $\gamma$ is horizontal if the (left) \emph{Maurer-Cartan form}

\begin{equation*}
c_{\gamma}\left( t \right):=( L_{\gamma\left( t \right)^{-1}})_{\ast} \gamma^{\prime} \left( t \right)
\end{equation*}
is in $\mathcal{H}$ for a.e.~$t$.
\end{definition}

The metric on $\mathcal{D}$ is defined by
\begin{align*}
	\langle u, v\rangle_x &:= \langle (L_{x^{-1}})_{\ast}u,(L_{x^{-1}})_{\ast}v\rangle_{\h} \qquad \text{for all } u,v\in\mathcal{D}_x.
\end{align*}
The length of a horizontal path $\gamma$ may  be computed as
\begin{equation}\label{eq-length}
 \ell(\gamma) := \int_0^{1} \sqrt{\langle \gamma^{\prime}(t), \gamma^{\prime}(t)\rangle_{\gamma(t)}}\,dt
	= \int_0^1 \sqrt{\langle c_\gamma(t), c_\gamma(t)\rangle_{\h}}\,dt
	:= \int_0^1  | c_\gamma \left( t \right)|_{\mathcal{H}}\,dt.
\end{equation}

\begin{example} For a Carnot group of step $2$ we can describe horizontal paths as follows. Suppose $\gamma(t)=(A(t), a(t))$ is an absolutely continuous path in $G$ with $A\left( t \right) \in \mathbb{R}^{d_{1}} \times \left\{ 0 \right\}$ and  $a\left( t \right) \in \left\{ 0 \right\} \times \mathbb{R}^{d_{2}}$. By Proposition~\ref{p.tt}

\begin{align*}
& (L_{\gamma(t)^{-1}})_{\ast}\gamma^{\prime}(t) = \gamma^{\prime}(t)
\\
& + \left(\mathbf{0}_{d_1},
			-\sum_{1\leqslant k< \ell \leqslant d_1}
	\alpha_{d_1+1}^{k, \ell}\omega_{k, \ell}\left( \gamma\left( t \right), \gamma^{\prime}\left( t \right) \right),\ldots,
		-\sum_{1\leqslant k< \ell \leqslant d_1}
	\alpha_{d_1+d_{2}}^{k, \ell}\omega_{k, \ell}\left( \gamma\left( t \right), \gamma^{\prime}\left( t \right) \right) \right)
\end{align*}
Recall that the path $\gamma$ is horizontal if $(L_{\gamma(t)^{-1}})_{\ast}\gamma^{\prime}(t)\in\h\times\{0\}$, and thus we have

\begin{align*}
& a(t)=(a_{d_{1}+1}(t), \ldots, a_{d_{1}+d_{2}}(t)),
\\
& a_{j}^{\prime}\left( t \right) =\sum_{1\leqslant k< \ell \leqslant d_1}
\alpha_{j}^{k, \ell}\omega_{k, \ell}\left( A\left( t \right), A^{\prime}\left( t \right) \right), j=d_{1}+1, ..., d_{1}+d_{2},
\end{align*}
for a.e.~$t \in [0, 1]$. That is,  a path $\gamma$  with $\gamma(0)=e$ is horizontal in a stratified group $G$ of step $2$ if it is of the form

\begin{equation}\label{e.horiz2}
	\gamma (t) = \left(A(t), \int_0^t Q^{(2)}(A(s), A^{\prime}(s))\,ds\right)
\end{equation}
	where $A(0)=0$, $Q^{(2)} = (Q_{d_1+1},\ldots,Q_{d_1+d_{2}})$ and
\[
Q_{j}\left( x, y \right)=\sum_{1\leqslant k<\ell\leqslant d_1}
	\alpha_{j}^{k, \ell}\omega_{k, \ell}\left( x^{(1)}, y^{(1)} \right).
\]
 The length of $\gamma$ is then given by
\begin{align*}
\ell(\gamma)
	= \int_0^1 |\gamma^{\prime}(s)|_{\gamma(s)}\,ds
	= \int_0^1  | c_\gamma \left( t \right)|_{\mathcal{H}}\,dt
	= \int_0^1 |A^{\prime}(s)|_\h\,ds .
\end{align*}
\end{example}

The group $G$ as a sub-Riemannian manifold may then be equipped with a natural left-invariant Carnot-Carath\'{e}odory distance.

\begin{definition}\label{d.CCdistance}
For any $x_{1}, x_{2} \in G$ the \emph{Carnot-Carath\'{e}odory distance} is defined as
\begin{align*}
\rho_{cc} (x_{1}, x_{2}):= &\inf \left\{ \ell\left( \gamma \right):\,\, \gamma : [0,1] \longrightarrow G \text{ is horizontal}, \gamma(0)=x_{1}, \gamma(1)=x_{2} \right\}.
\end{align*}
We denote by
\[
d_{cc} \left( x \right):= \rho_{cc} \left( e, x\right)
\]
the corresponding norm.
\end{definition}

The assumption that $\h$ generates the Lie algebra in \eqref{e.Stratification} means that any basis of $\h$ will satisfy H\"{o}rmander's condition. Therefore any two points in $G$ can be connected by a horizontal path by the Chow--Rashevskii theorem, and the Carnot-Carath\'{e}odory distance is finite on $G$. By \cite[Theorem 5.15]{BonfiglioliLanconelliUguzzoniBook} the Carnot-Carath\'{e}odory distance is realized, that is, for any two points in $G$ there is a horizontal path connecting those points which is a geodesic, so the infimum in Definition \ref{d.CCdistance} is actually a minimum.

The Carnot-Carath\'{e}odory distance is just one of the distances on $G$ which is left-invariant and homogeneous with respect to dilations.

\begin{definition}[Homogeneous distances and norms]\label{d.HomogeneousDistance}
	A \emph{homogeneous distance} on $G$ is a continuous, left-invariant distance $\rho: G \times G \longrightarrow [0,\infty)$  such that
\begin{align*}
& \rho\left( D_{a} x, D_{a} y \right)=a\rho\left(  x, y \right)
\end{align*}
for any $a>0$ and $x, y \in G$.
The corresponding \emph{homogeneous norm} will be denoted by $d\left( x \right):=\rho\left(  e, x \right)$.
\end{definition}

It may be shown that all homogeneous norms on $G$ are equivalent.
\begin{prop}\label{p.5.1.4}\cite[Proposition 5.1.4]{BonfiglioliLanconelliUguzzoniBook}
	Let $d$ be any homogeneous norm on $G$. Then there exists a constant $c>0$ so that
	\[ c^{-1}|x|_G \leqslant d(x)\leqslant c|x|_G \]
	where
	\[ |x|_G := \left(\sum_{j=1}^r \|x^{(j)}\|_{\R^{d_j}}^{2r!/j}\right)^{1/2r!}. \]
\end{prop}

Similarly, all homogeneous distances satisfy the following.
\begin{prop}\label{p.NormEquiv} \cite[Proposition 5.15.1]{BonfiglioliLanconelliUguzzoniBook}
Let $\rho$ be any left-invariant homogeneous distance on $G$. Then, for any compact set $K \subset G$, there exists a constant $c_{K}>0$ such that
\begin{align*}
	c_{K}^{-1} \Vert x-y \Vert_{\mathbb{R}^{N}} \leqslant \rho\left( x, y \right) \leqslant c_{K} \Vert x-y \Vert_{\mathbb{R}^{N}}^{1/r}
\end{align*}
where $r$ is the step of $G$.
\end{prop}

Therefore the topology of a homogeneous Carnot group $\left( \mathbb{R}^{N}, \star \right)$ with respect to the Carnot--Carath\'{e}odory distance (or any other homogeneous distance) coincides with the Euclidean topology of $\mathbb{R}^{N}$. More precisely, the categories of open, closed, bounded, or compact  sets coincide in these two topologies \cite[Proposition 5.15.4]{BonfiglioliLanconelliUguzzoniBook}. There is a huge literature on the subject, starting with Chow and Rashevsky. More details on homogeneous 
distances can be found in \cite[Sections 5.1 and 5.2]{BonfiglioliLanconelliUguzzoniBook}, and many references can be found in the bibliography of that text.

Recall that if $G$ is a general Lie group with a Lie algebra $\mathfrak{g}$ equipped with an inner product $\langle \cdot, \cdot \rangle$, we can define the corresponding left-invariant Riemannian distance on $G$. Then the map $(L_x)_{\ast}v=dL_{x}\left( v \right): G \times \mathfrak{g} \longrightarrow TG$ as introduced in Notation~\ref{n.2.11}

\begin{align*}
dL_{\cdot} \left( \cdot \right): G \times \mathfrak{g} &\longrightarrow TG, 
\\
\left( x, v \right) &\longmapsto \left( L_{x} \right)_{\ast} v \in T_{x} G
\end{align*}
is smooth, and thus $(L_{x})_{\ast} v$ is locally Lipschitz as a mapping in $(x,v)$ with respect to the product topology on $TG$.

For the present paper, we assume that only $\h$ is equipped with an inner product $\langle\cdot, \cdot\rangle_\h$ with $|\cdot|_\h$ being the associated norm on $\h\cong\R^{d_1}$. We will need an analogous Lipschitz property in this setting. We use Proposition \ref{p.tt} to prove the following statement.

\begin{prop}
For any compact domain $D\subset G\times  \mathcal{H}$, there exists a constant $C_D>0$ such that for any $(x,v), (y,u)\in D$,
\begin{equation}\label{eq-cont-L}
\Vert(L_x)_{\ast}u-(L_y)_{\ast}v\Vert_{\R^N} \leqslant C_D\left(\vert u-v\vert_{\mathcal{H}}+\rho(x,y)\right),
\end{equation}
	where $\rho$ can be the Euclidean norm on $\mathbb{R}^{N}$ or any left-invariant homogeneous distance, and the constant $C_D$ depends only on $D$ and the choice of $\rho$. Moreover, for any $x, y\in K$, a compact subset of $G$, and $v\in \h$

\begin{equation}\label{lemma-transla-lips}
\|(L_x)_{\ast} v-(L_y)_{\ast} v \|_{\R^N}\leqslant C_K |v|_\h \|x-y\|_{\R^N}.
\end{equation}

\end{prop}

\begin{proof}
	By Proposition \ref{p.tt}
\begin{align*}
	( L_x )_{\ast}\left( u \right)&-\left( L_y \right)_{\ast}\left( v \right)
		=u-v\\
		&+\left( \mathbf{0}_{d_{1}},  \sum_{(k, \ell)\in I_{d_1+1}}C_{d_{1}+1}^{k, \ell}\left( x \right) \omega_{k, \ell}\left( x, u \right), \ldots, \sum_{(k, \ell)\in I_N}C_{N}^{k, \ell}\left( x \right) \omega_{k, \ell}\left( x, u \right)  \right)
\\
&
	-\left( \mathbf{0}_{d_{1}},  \sum_{(k, \ell)\in I_{d_1+1}}C_{d_{1}+1}^{k, \ell} \left( y \right)\omega_{k, \ell}\left( y, v \right), \ldots, \sum_{(k, \ell)\in I_N}C_{N}^{k, \ell} \left( y \right)\omega_{k, \ell}\left( y, v \right)  \right).
\end{align*}
We have that
\[
\vert \omega_{k, \ell}\left( x, u \right)-\omega_{k, \ell}\left( y, v \right) \vert \leqslant \Vert v \Vert_{\mathbb{R}^{N}}\Vert x-y \Vert_{\mathbb{R}^{N}}+\Vert x \Vert_{\mathbb{R}^{N}}\Vert u-v \Vert_{\mathbb{R}^{N}}
\]
and so
\begin{multline*}
\Vert(L_x)_{\ast}\left( u \right)-(L_y)_{\ast}\left( v \right) \Vert_{\R^N}
\\
\leqslant\Vert u-v \Vert_{\mathbb{R}^{N}}+\max_{j, k, \ell}\{\vert C_{j}^{k, \ell}\left( x \right)\vert,|C_j^{k,\ell}(y)|\} \left(\Vert v \Vert_{\mathbb{R}^{N}}\Vert x-y \Vert_{\mathbb{R}^{N}}+\Vert x \Vert_{\mathbb{R}^{N}}\Vert u-v \Vert_{\mathbb{R}^{N}}\right).
\end{multline*}
	For $u, v \in \mathcal{H}$ we have $\Vert u-v \Vert_{\mathbb{R}^{N}}=\Vert u-v \Vert_{\R^{d_1}}\leqslant C|u-v|_\h$, and so \eqref{eq-cont-L} (with $\rho=\|\cdot\|_{\R^N}$) and \eqref{lemma-transla-lips} follow. Inequality \eqref{eq-cont-L} for a general left-invariant homogeneous distance $\rho$ follows from Proposition \ref{p.NormEquiv}.
\end{proof}

The next lemma is a version of Gr\"{o}nwall's lemma in the sub-Riemannian setting, which says that if $\sigma$ and $\gamma$ are horizontal paths starting at the origin whose Maurer-Cartan forms $c_{\sigma}$ and $c_{\gamma}$ are close in $L^1$, then the paths cannot get too far away from each other.

\begin{lemma}[Gr\"{o}nwall's lemma]\label{l.Gronwall2}
Let $G$ be a homogeneous Carnot group modeled on $\mathbb{R}^{N}$. Suppose  $\varepsilon >0$,  and $\sigma, \gamma:[0,1] \to G$ are horizontal paths such that $\sigma(0)=\gamma(0)=e$,
\begin{equation*}
\int_0^1|c_ \sigma(t)|_\h dt <\infty,
\end{equation*}
and
	\[ \int_0^1 |c_{\sigma}\left( t \right)-c_{\gamma}\left( t \right)|_{\h}\, dt < \varepsilon.
	\]	
	Then there exists a constant $C=C(\|c_\sigma\|_{L^1})<\infty$  such that
\begin{equation}\label{eq-gronwall-conclu}
	\rho_{cc}\left( \sigma(1),\gamma(1) \right) < C\varepsilon.
\end{equation}
\end{lemma}

\begin{proof} First note that from the assumptions there exists a compact set $K$ that contains both paths $\gamma\left([0, 1]\right)$ and $\sigma\left([0, 1]\right)$ entirely. By Proposition \ref{p.NormEquiv}, it then suffices to prove that $\|\sigma(1)-\gamma(1)\|_{\R^N}\leqslant C^{\prime}\varepsilon$ for some constant $C^{\prime}<\infty$. In the following estimates $C$ will be a constant that depends only on $K$, which in turn depends on $\|c_\sigma\|_{L^1}$ but may vary from line to line.
	
To begin, taking the derivative of $\gamma$ and $\sigma$ in the ambient Euclidean space $\mathbb{R}^{N}$, we have that

\begin{align*}
\gamma^{\prime}\left( t \right)-\sigma^{\prime}\left( t \right)
&
=\left(L_{\gamma\left( t \right)}\right)_{\ast}c_{\gamma}\left( t \right)
	-\left(L_{\sigma\left( t \right)}\right)_{\ast}c_{\sigma}\left( t \right)
\\
&=\left(L_{\gamma\left( t \right)}\right)_{\ast}\left(c_{\gamma}\left( t \right)-c_{\sigma}\left( t \right) \right)
+\left(L_{\gamma\left( t \right)}-L_{\sigma\left( t \right)}\right)_{\ast}c_{\sigma}\left( t \right).
\notag
\end{align*}
Hence 
\begin{align*}
\Vert \gamma^{\prime}\left( t \right)-\sigma^{\prime}\left( t \right)\Vert_{\mathbb{R}^{N}}
	&\leqslant
\Vert \left(L_{\gamma\left( t \right)}\right)_{\ast}\left(c_{\gamma}\left( t \right)-c_{\sigma}\left( t \right) \right)\Vert_{\mathbb{R}^{N}}
+ \Vert \left(L_{\gamma\left( t \right)}-L_{\sigma\left( t \right)}\right)_{\ast}c_{\sigma}\left( t \right) \Vert_{\mathbb{R}^{N}}
\\
& \leqslant
C\left(|c_{\gamma}\left( t \right)-c_{\sigma}\left( t \right) |_{\mathcal{H}}
+ |c_{\sigma}\left( t \right) |_\h \|{\sigma}\left( t \right)- {\gamma}\left( t \right)\|_{\R^N}\right),
\end{align*}
where in the second inequality we have applied \eqref{eq-cont-L} and \eqref{lemma-transla-lips}.

Now let
\[
F\left( t \right):=\Vert\gamma\left( t \right)- \sigma\left( t \right)\Vert_{\mathbb{R}^{N}}^{2}.
\]
Then for a.e.~$t$
\begin{align*}
\frac{dF}{dt}
		&= 2\langle \gamma^{\prime}\left( t \right)-\sigma^{\prime}\left( t \right), \gamma\left( t \right)- \sigma\left( t \right)\rangle_{\mathbb{R}^{N}}
\leqslant 2 \Vert \gamma^{\prime}\left( t \right)-\sigma^{\prime}\left( t \right)\Vert_{\mathbb{R}^{N}}
\Vert \gamma\left( t \right)- \sigma\left( t \right)\Vert_{\mathbb{R}^{N}}
\\
	&\leqslant 2C\left(\vert c_{\gamma}\left( t \right)-c_{\sigma}\left( t \right) \vert_\h + |c_{\sigma}\left( t \right) |_\h\|\gamma(t)-\sigma(t)\|_{\R^N}\right)
\Vert \gamma\left( t \right)- \sigma\left( t \right)\Vert_{\mathbb{R}^{N}}.
\end{align*}

Let $t_0:=\sup_{t\in[0,1]}\{\sigma(t)=\gamma(t)\}$. If $t_0=1$, then we have that $\sigma(1)=\gamma(1)$ and \eqref{eq-gronwall-conclu} holds automatically. If $t_0\in [0,1)$, we can consider new paths $\sigma[t_0, 1]$ and $\gamma[t_0, 1]$ starting at $\sigma(t_0)=\gamma(t_0)$ and show that their endpoints are close. With this argument we can then assume that $F(t)\not=0$ for all $t\in (0, 1)$.
Consider $G(t):= \|\gamma(t)-\sigma(t)\|_{\R^N} =\sqrt{F(t)}$, then
\begin{align*}
\frac{dG}{dt}
		&=\frac{1}{2G(t)}F'(t)\leqslant  C\vert c_{\gamma}\left( t \right)-c_{\sigma}\left( t \right) \vert_\h  + C\,  |c_{\sigma}\left( t \right) |_\h G(t)
\end{align*}

We have
\begin{equation}\label{eq-Gronwall-a}
\frac{dG}{dt} \leqslant A(t)+B(t)G\left( t \right),
\end{equation}
where $A(t)=C\vert c_{\gamma}\left( t \right)-c_{\sigma}\left( t \right) \vert_\h$ and $B(t)=C |c_\sigma(t)|_\h$.
Let $b(t)=\int_0^t B(s)ds={C}\int_0^t |c_\sigma(s)|_\h\, ds$. Then \eqref{eq-Gronwall-a} can be written as
\[
	\frac{d}{dt}\left( e^{-b(t)}G\left( t \right) \right) \leqslant A(t) e^{-b(t)}.
\]
Clearly $G(0)=0$. It follows that
\[
	e^{-b(t)}G\left( t \right)\leqslant \int_0^t A(s)e^{-b(s)}\,ds.
\]
In particular for $t=1$ we have that
\begin{align*}
	G\left( 1 \right)=\Vert\gamma\left( 1 \right)- \sigma\left( 1 \right)\Vert_{\mathbb{R}^{N}}
	&\leqslant C e^{b(1)}
	\int_0^1 |c_\gamma(t)-c_\sigma(t)|_\mathcal{H}e^{-b(t)} \,dt \\
	&\leqslant C e^{C\|\sigma\|_{L^1}}\int_0^1 |c_\gamma(t)-c_\sigma(t)|_\mathcal{H} \,dt
	< C\varepsilon.
\end{align*}
\end{proof}

\begin{remark}
We consider the specific homogeneous distance $\rho_{cc}$ in this version of Lemma \ref{l.Gronwall2}, but in light of its proof and Proposition \ref{p.NormEquiv}, it is clear that this result could be stated with $\rho_{cc}$ replaced by any left-invariant homogeneous distance on $G$.
\end{remark}

\begin{remark}
This estimate appears in various places in the literature (see for example \cite[Lemma 6.7]{Breuillard2014} or \cite[Lemma 4.2.5]{LeDonneNotes2018}) for general Lie groups. However, in the general Lie group setting, it is difficult to make sense of comparing vectors in different tangent spaces in the absence of the unifying context of the ambient space $\R^N$.

That being said, the same proof as the one given here works for matrix Lie groups equipped with a left-invariant Riemannian distance $\rho$ under the (standard) assumption that the Lie bracket satisfies the continuity assumption
\begin{equation*}
\vert [ A, B ]\vert_{\mathfrak{g}} \leqslant M \vert A \vert_{\mathfrak{g}}\vert B\vert_{\mathfrak{g}}
\end{equation*}
for any $A, B \in \mathfrak{g}$.
\end{remark}

\begin{prop}\label{prop-L} Given a Carnot group $G$, define
\begin{align*}
	L: G &\longrightarrow \mathcal{H}:=\mathcal{V}_{1},
\\
	x &\longmapsto \int_{0}^{1} c_{\gamma}\left( t \right)dt,
\end{align*}
where $\gamma$ is any horizontal path  such that $\gamma\left(0\right)=e$ and $\gamma \left(1\right)=x$. Then
\[
L\left( x \right)= P_{\mathcal{H}}\left( \log \left( x\right) \right),
\]
where $P_{\mathcal{H}}$ is the projection onto $\mathcal{H}$, and in particular $L$ is well-defined independent of the choice of $\gamma$. Additionally, $L$ is continuous as a map from $\left( G, \rho_{cc} \right)$ to $\left( \mathcal{H}, \langle \cdot, \cdot \rangle_{\mathcal{H}} \right)$.
\end{prop}

\begin{proof} First recall that $\log$ and $\exp$ are global diffeomorphisms in this setting. Therefore we can use the Baker-Campbell-Dynkin-Hausdorff formula \eqref{e.BCDH} to see that for any $x_{1}, x_{2} \in G$ we can find $\X_{1}, \X_{2} \in \mathfrak{g}$ such that $x_1=e^{\X_{1}}$ and $x_{2}=e^{\X_{2}}$, and thus
\begin{align*}
	P_{\mathcal{H}}\left( \log\left( x_{1}x_{2} \right)\right)
		&=P_{\mathcal{H}}\left( \log\left( e^{\X_{1}} e^{\X_{2}} \right)\right)
			=P_{\mathcal{H}}\left(\X_{1}\right)+P_{\mathcal{H}}\left(\X_{2}\right)\\
		&=P_{\mathcal{H}}\left( \log\left( e^{\X_{1}} \right)\right)+P_{\mathcal{H}}\left( \log\left(  e^{\X_{2}} \right)\right)
			=P_{\mathcal{H}}\left( \log\left( x_{1} \right)\right)+P_{\mathcal{H}}\left( \log\left(  x_{2} \right)\right).
\end{align*}
In particular, this implies that
\[
P_{\mathcal{H}}\left( \log\left( x_{1} x_{2} \right)\right)=P_{\mathcal{H}}\left( \log\left( x_{2} x_{1} \right)\right).
\]
Additionally,
\[
P_{\mathcal{H}}\left( \log\left( x^{-1} \right)\right)=P_{\mathcal{H}}\left( \log\left( e^{-\X} \right)\right)=-P_{\mathcal{H}}\left( \X \right)=-P_{\mathcal{H}}\left( \log  x  \right).
\]

Now let $\gamma$ be any horizontal path such that $\gamma\left(0\right)=e$ and $\gamma \left(1\right)=x$. Using the above observations for  $x_{1}=\gamma\left( t \right)$ and $x_{2}=\gamma\left( t+\varepsilon \right)$, we see that
\begin{multline*}
P_{\mathcal{H}}\left( \log \left( \gamma \left(t+\varepsilon \right) \right) \right)-P_{\mathcal{H}}\left( \log \left( \gamma \left(t \right) \right) \right)
=
P_{\mathcal{H}}\left( \log \left( \gamma \left(t \right)^{-1} \right) \right)+P_{\mathcal{H}}\left( \log \left( \gamma \left(t+\varepsilon \right) \right) \right)
\\
 =P_{\mathcal{H}}\left( \log \left( \gamma \left(t \right)^{-1}  \gamma \left(t+\varepsilon \right) \right) \right)-P_{\mathcal{H}}\left( \log \left( \gamma \left(t \right)^{-1}  \gamma \left(t \right) \right) \right),
\end{multline*}
since $P_{\mathcal{H}}\left( \log \left(  e \right) \right)=0$. Therefore
\begin{align*}
	\frac{d}{dt}P_{\mathcal{H}}\left( \log \left( \gamma \left(t \right) \right) \right)
	&= \frac{d}{d\varepsilon}\bigg|_0
		P_{\mathcal{H}}\left( \log \left( \gamma \left(t+\varepsilon \right) \right) \right) \\
	&=\left( P_{\mathcal{H}} \circ \log \right)_{\ast} \left( \left( L_{\gamma\left( t \right)}\right)_{\ast}^{-1} \gamma^{\prime}\left( t \right) \right)
\\
& =\left( P_{\mathcal{H}} \circ \log \right)_{\ast} \left( c_{\gamma}\left( t \right) \right)=c_{\gamma}\left( t \right),
\end{align*}
for a.e.\,$t$ in $[ 0, 1]$, since  $P_{\mathcal{H}} \circ \log: G \longrightarrow \mathcal{H}$, and its differential (pushforward)  $d\left( P_{\mathcal{H}} \circ \log \right): \mathfrak{g} \longrightarrow \mathcal{H}$ is the identity on the horizontal space $\mathcal{H}$. In particular,
\begin{align*}
& P_{\mathcal{H}}\left( \log \left( x \right) \right)=\int_{0}^{1} \frac{d}{dt}P_{\mathcal{H}}\left( \log \left( \gamma \left(t \right) \right) \right) dt=\int_{0}^{1}c_{\gamma}\left( t \right) dt \in \mathcal{H}
\end{align*}
for any horizontal path $\gamma$ connecting $e$ and $x$.

Similarly, if we have $x_{1}, x_{2} \in G$ and any horizontal path $\gamma$ connecting $x_{1}$ and $x_{2}$, we see that
\begin{align*}
& P_{\mathcal{H}}\left( \log \left( x_{1} \right) \right)-P_{\mathcal{H}}\left( \log \left( x_{2} \right) \right)=\int_{0}^{1}c_{\gamma}\left( t \right) dt.
\end{align*}
Thus,
\begin{align*}
& \vert P_{\mathcal{H}}\left( \log \left( x_{1} \right) \right)-P_{\mathcal{H}}\left( \log \left( x_{2} \right) \right)\vert_{\mathcal{H}}=\left\vert \int_{0}^{1}c_{\gamma}\left( t \right) dt \right\vert_\h
\leqslant \int_{0}^{1}\vert c_{\gamma}\left( t \right)\vert_\h dt=\ell\left( \gamma \right),
\end{align*}
and taking the infimum over all such horizontal paths $\gamma$ gives
\[
\left\vert L\left( x_{1} \right)-L\left( x_{2} \right) \right\vert_{\mathcal{H}}\leqslant \rho_{cc}\left( x_{1}, x_{2} \right),
\]
which implies continuity. \end{proof}

The significance of the map $L$ is that unlike in the case of Lie groups equipped with a Riemannian metric, we only have metric on the horizontal space $\mathcal{H}$, but $\log$ does not respect the horizontal structure. This is a fundamental difference from the techniques used in \cite[Proposition 5.2]{Versendaal2019b}, and one can think of the map $L$ as a horizontal logarithmic map.

\section{Large deviations}\label{s.LDP}

Now we return to our main result, Theorem~\ref{t.main}. Suppose $\left\{ X_{i} \right\}_{i=1}^{\infty}$ is a sequence of i.i.d.~random variables in $\mathcal{H}$ and define
	\[
S_{n}:=\exp(X_{1})\star\cdots\star \exp(X_{n}).
\]
Note that for a dilation $D_{a}: G \longrightarrow G$ by \eqref{e.Dilations} we have
\[
D_{a}S_{n}=\prod_{i=1}^{n}\star D_{a} \exp(X_{i}) =\prod_{i=1}^{n}\star \exp( \delta_{a}X_{i}).
\]

The proof consists of several steps. First in Section~\ref{ss.vector-ldp} we will consider a partial linearization of the $n$-fold group multiplication by partitioning $D_{\frac{1}{n}} S_{n}$ into a fix number of blocks. This linearization naturally lives in the product space $\mathcal{H}^{m}$, and so we will prove a large deviations principle for sequences of random vectors in $\mathcal{H}^m$. In Section~\ref{ss.approx} we will use the linearization to find a family of exponentially good approximations to $\{D_{\frac{1}{n}} S_{n}\}$. Then in Section~\ref{ss.group-ldp} we combine these results with the contraction principle, Theorem~\ref{thm-contract}, to prove Theorem~\ref{t.main}.

\subsection{Vector space LDP}
\label{ss.vector-ldp}

For a fixed $m\in\{1,\ldots,n\}$, we partition $D_{\frac{1}{n}} S_{n}$ into $m$ pieces as follows. For $k=0,\ldots,m-1$, let $n_{k}=k\lfloor\frac{n}{m}\rfloor$ and $n_m=n$, and for $k=1, \ldots, m$ we define
\[
S^{m,k}_n:=\prod_{i=n_{k-1}+1}^{n_{k}}\star \exp( \delta_{\frac1n}X_{i}).
\]
Now let
	\begin{align*}
Y_n^{m,k}:=L(S_n^{m,k})
		&=L\left(\exp(\delta_{\frac1n}X_{n_{k-1}+1})\star\cdots\star\exp(\delta_{\frac1n}X_{n_{k}})  \right)\\
	&=\frac{1}{n}\left(X_{n_{k-1}+1}+\cdots + X_{n_{k}}\right)
\in \mathcal{H},
\end{align*}
where $L$ is the map defined in Proposition \ref{prop-L}, and take
\begin{equation}\label{eq-Y-m}
Y_{n}^m:=(Y_n^{m,1},\dots, Y_n^{m,m} )\in \mathcal{H}^m.
\end{equation}
It will be useful later to note that, taking $d=\lfloor\frac{n}{m}\rfloor$ so that $n=dm+r$ for some $r\in\{0,1\ldots,m-1\}$, we have that for $k\in [m-1]$, each $Y^{m,k}_n$ consists of $d$ steps, and $Y^{m,m}_n$ consists of $d+r$ steps. Now we prove the following large deviation principle for $\{Y_{n}^m\}\in\h^m$.

\begin{prop}\label{p.Hldp}
Suppose $\{X_i\}_{i=1}^\infty $ are i.i.d.~mean 0 random variables in $\h$ such that
\[
\Lambda(\lambda) := \Lambda_X(\lambda)
	:= \log \E \left[\exp\left(\langle \lambda, X_{1} \rangle_{\mathcal{H}}\right)\right]
	\]
exists for all $\lambda\in\h$. Fix $m\in\mathbb{N}$. Then for any closed $F\subset \h^m$ and open $O\subset \h^m$ we have that
\[
\limsup_{n \to \infty}\frac1n \log\p \left(Y_n^{m}\in F \right)
\leqslant - \inf_{u\in F} I_m(u)
\]
and
\[
\liminf_{n \to \infty}\frac1n \log\p \left(Y_n^m\in O\right)
	\geqslant -\inf_{u\in O} I_m(O),
\]
where for $u=(u_{1},\ldots,u_m)\in\h^m$
\begin{equation}\label{eq-I_m}
I_m(u):= \frac{1}{m}\sum_{k=1}^m \Lambda^{\ast}(mu_{k})
\end{equation}
and $\Lambda^{\ast}$ is the Legendre transform,
\begin{equation}\label{e.FL} \Lambda^{\ast}(u) = \sup_{v\in\h} \,\left(\langle v,u\rangle_\h - \Lambda(v)\right).
\end{equation}
\end{prop}

\begin{proof}
	To prove the upper bound, we first note that by following the proof of the upper bound in the classical Cram\'{e}r's theorem (see for example \cite[p. 37]{DemboZeitouniBook2010}, or Proposition \ref{prop-ldp-hm} in the appendix) one may show that for any closed set $F\subset \h^m$
\begin{align}
& \limsup_{n \to \infty}\frac{1}{n} \log\p \left(Y_n^{m}\in F \right)  \label{e.prop-ldp-hm}
\\
& \leqslant -\inf_{u\in F} \sup_{\lambda\in \mathcal{H}^m} \left\{ \langle \lambda, u\rangle_{\mathcal{H}^m} -\limsup_{n\to\infty}\frac1n \log \E \left[\exp\left(n\langle \lambda, Y_n^{m}\rangle_{\mathcal{H}^m} \right)\right] \right\}.
\notag
\end{align}
	Now for $\lambda=(\lambda_{1},\ldots,\lambda_m)\in\h^m$
\begin{align*}
\E \left[\exp\left(n\langle \lambda, Y_n^{m}\rangle_{\mathcal{H}^m} \right)\right]
		&= \prod_{k=1}^{m}\E \left[\exp\left(n\langle \lambda_{k}, Y_n^{m,k}\rangle_{\mathcal{H}} \right)\right] \\
		&= \prod_{k=1}^{m} \prod_{i=n_{k-1}+1}^{n_{k}}
		\E \left[\exp\left(\langle \lambda_{k}, X_i\rangle_{\mathcal{H}} \right)\right]
		= \prod_{k=1}^{m}\exp( \Lambda(\lambda_{k}))^{n_{k}-n_{k-1}}.
	\end{align*}
Again letting $d=\lfloor\frac{n}{m}\rfloor$ so that $n=dm+r$ for some $r\in\{0,1\ldots,m-1\}$, we have that $n_{k}-n_{k-1}=d$ for all $k\in[m-1]$ and $n_m-n_{m-1}=n-(m-1)d=n-md+d = d+r$. Thus we may write
\[
\E \left[\exp\left(n\langle \lambda, Y_n^{m}\rangle_{\mathcal{H}^m} \right)\right]
	= \left(\prod_{k=1}^{m}\exp(\Lambda(\lambda_{k}))\right)^{d}
			\exp(\Lambda(\lambda_m))^r
\]
and so
\begin{align*}
\limsup_{n\to\infty} \frac1n \log \E \left[\exp\left(n\langle \lambda, Y_n^{m}\rangle_{\mathcal{H}^m} \right)\right]
		&=  \limsup_{n\to\infty}\frac{\lfloor n/m\rfloor}{n}\sum_{k=1}^m \Lambda(\lambda_{k}) 
			+ \frac{r}{2n} \Lambda(\lambda_m) \frac{r}{n} \Lambda(\lambda_m) \\
		&= \frac{1}{m}\sum_{k=1}^m \Lambda(\lambda_{k}).
\end{align*}
So for all $u=(u_{1},\ldots,u_m)\in\h^m$ we have
\begin{align*}
	 & \sup_{\lambda\in \mathcal{H}^m} \left\{ \langle \lambda, u\rangle_{\mathcal{H}^m}  - \limsup_{n\to\infty}\frac1n \log \E \left[\exp\left(n\langle \lambda, Y_n^{m}\rangle_{\mathcal{H}^m} \right)\right] \right\} \\
	&= \sup_{\lambda\in \mathcal{H}^m} \left\{ \sum_{k=1}^m \langle \lambda_{k}, u_{k}\rangle_{\mathcal{H}}  -
		\frac{1}{m}\sum_{k=1}^m \Lambda(\lambda_{k}) \right\} \\
	&= \frac{1}{m} \sup_{\lambda\in \mathcal{H}^m} \left\{ \sum_{k=1}^m \langle \lambda_{k}, mu_{k}\rangle_{\mathcal{H}}  -
		\sum_{k=1}^m \Lambda(\lambda_{k}) \right\} \\
	&\leqslant \frac{1}{m}   \sum_{k=1}^m \sup_{\lambda_{k}\in \mathcal{H}} \left\{\langle \lambda_{k}, mu_{k}\rangle_{\mathcal{H}}  - \Lambda(\lambda_{k}) \right\}
	= \frac{1}{m} \sum_{k=1}^m \Lambda^{\ast}(mu_{k}).
\end{align*}

For the lower bound, as usual it suffices to prove that for any $u=(u_{1},\ldots,u_m)\in\h^m$ and $\varepsilon>0$ we have
\[
\liminf_{n \to \infty}\frac1n \log\p \left(Y_n^m\in B(u,\varepsilon)\right)
	\geqslant - \frac{1}{m}\sum_{k=1}^m \Lambda^{\ast}(mu_{k}).
\]
Choose $\delta>0$ sufficiently small so that
\[
B(u_{1},\delta)\times\cdots\times B(u_m,\delta)\subset B(u,\varepsilon).
\]
Then
\begin{align*}
\log\p \left(Y_n^m\in B(u,\varepsilon)\right)
	&\geqslant \log\p\left(Y_n^{m,1}\in B(u_{1},\delta),\ldots,Y_n^{m,m}\in B(u_m,\delta)\right) \\
	&= \log \prod_{k=1}^m \p\left(Y_n^{m,k}\in B(u_{k},\delta)\right) \\
	&= \sum_{k=1}^m \log\p\left(Y_n^{m,k}\in B(u_{k},\delta)\right).
\end{align*}
	Recall that for each $k\in[m-1]$, $Y_n^{m,k}= \frac{1}{n}(X_{n_{k-1}+1}+\cdots + X_{n_{k-1}+\lfloor n/m\rfloor})$, and so again by the classical Cram\'{e}r's theorem
\[ \liminf_{n \to \infty}\frac1n \log\p\left(Y_n^{m,k}\in B(u_{k},\delta)\right)
	\geqslant -\frac{1}{m}\Lambda^{\ast}(mu_{k}),\]
	and the $k=m$ case can be dealt with similarly as was done in the upper bound case, yielding the desired result.

\end{proof}

\subsection{Exponentially good approximations}\label{ss.approx}

For a fixed $m \in \mathbb{N}$, define the map $\Psi_m:\mathcal{H}^m\to G$ by
\begin{equation}\label{e.Psim}
\Psi_m(u_{1}, \ldots, u_{m}) := \exp(u_{1})\star\cdots \star\exp(u_{m}).
\end{equation}

\begin{example}[Example~\ref{ex.Heisenberg} revisited]
In this case $\Psi_m:\mathcal{H}^m\to G$  can be viewed as the map $\Psi_m:\left( \mathbb{R}^{2d} \right)^m\to \mathbb{R}^{2d} \times \mathbb{R}$ with  $u_{i}=\left( x_{i}, y_{i} \right)$

\[
\Psi_m\left( \left( x_{1}, y_{1}\right), ..., \left( x_{m}, y_{m}\right) \right)=\left( \sum_{i=1}^{n} x_{i}, \sum_{i=1}^{n} y_{i}, \sum_{1\leqslant i< j\leqslant n} \left( x_{i}, y_{i} \right) A \left( x_{j}, y_{j} \right)^{T}\right).
\]
\end{example}
In this section we show that $\{\Psi_m(Y_n^m)\}_{m \in \mathbb{N}}$ are exponentially good approximations of $\{D_{\frac{1}{n}} S_{n}\}$, as in Definition \ref{d.exp-approx},
considering the following cases separately: Carnot groups of step $2$  with random vectors $X_i$ having a sub-Gaussian distribution on $\h$, general Carnot groups where $X_i$ are bounded almost surely, and general Carnot groups where $X_i$ are Gaussian. We will rely on the standard argument that if
\begin{equation}\label{e.standard1}
\E[ e^{\pm\lambda Z}] \leqslant C(|\lambda|),
\end{equation}
for some $\lambda\in\mathbb{R}$, then for any $t>0$
\begin{align}
	\notag \p(|Z|\geqslant t ) &= \p(Z\geqslant t) + \p(-Z\geqslant t) \\
		&\label{e.standard}= \p(e^{\lambda Z} \geqslant e^{\lambda t}) + \p(e^{-\lambda Z} \geqslant e^{\lambda t})
		\leqslant 2e^{-\lambda t}C(|\lambda|)
\end{align}
by Markov's inequality.

As it turns out, the general Gaussian case is the most technically challenging. This is due to the combination of the general group setting and the unbounded support of the Gaussian. We thus start with the other cases, which will allow us to introduce some of the necessary ideas.

\subsubsection{Step $2$ groups with sub-Gaussian distributions}

Recall that a mean zero random variable $X$ is \emph{sub-Gaussian} if there is a $k>0$ such that
\[
\E[e^{\lambda X}] \leqslant e^{k\lambda^2} \qquad \text{ for all } \lambda\in\mathbb{R}.
\]
A random vector in $\mathbb{R}^n$ is \emph{sub-Gaussian} if the one-dimensional marginals $\langle x, X\rangle$ are sub-Gaussian for all $x\in \mathbb{R}^n$.

Before proceeding further, we take the opportunity to further expand on some comments made in the introduction of this paper.

\label{r.comp}\begin{remark}
We look here at the step two case to illustrate our comments in Section~\ref{s.intro} comparing our results with \cite{BaldiCaramellino1999}. Let $X_1,X_2,\ldots$ and $Z_1,Z_2,\ldots$ be independent random variables on $\mathcal{H}$ and $\mathcal{V}$, respectively, and take
\[\mathcal{S}_n := \exp(X_1,Z_1)\star\cdots\star\exp(X_n,Z_n). \]
Then
\[ D_{1/n}\mathcal{S}_n\star(D_{1/n}S_n)^{-1}
	= \left( 0,\frac{1}{n^2}(Z_1+\cdots+Z_n)\right)\]
and so by Proposition~\ref{p.NormEquiv}
\begin{align*}
	P(d_{cc}(D_{1/n}\mathcal{S}_n\star(D_{1/n}S_n)^{-1})>\delta)
		\sim P\left(\frac{c}{n^2}\|Z_1+\cdots+Z_n\|>\delta\right),
\end{align*}
	where on the right hand side this is just the standard Euclidean distance from the center. Suppose the distribution of each random variable $Z$ is sub-Gaussian with parameter $k$. Then $Z_1+\cdots+Z_n$ is sub-Gaussian with parameter $nk$, and so

\begin{align*}
P\left(\frac{c}{n^2}\|Z_1+\cdots+Z_n\|>\delta\right)
	&= P\left(\|Z_1+\cdots+Z_n\|>\frac{\delta n^2}{c}\right)\\
	&\leqslant 2 e^{-\delta n^2\lambda/c}e^{nk\lambda^2}
\end{align*}
for arbitrary $\lambda$.
Keeping $\lambda$ a fixed constant, we may say that
\begin{align*}
	\lim_{n\to\infty} \frac{1}{n} \log P\left(\frac{c}{n^2}\|Z_1+\cdots+Z_n\|>\delta\right)
	&\leqslant \lim_{n\to\infty} \frac{1}{n} \log 2 e^{-\delta n^2\lambda/c}e^{nk\lambda^2}\\
	&= \lim_{n\to\infty} \frac{1}{n} \left(\log 2 -\frac{\delta n^2\lambda}{c} +
		nk\lambda^2\right)
	=-\infty
\end{align*}
	and thus $\{D_{1/n}\mathcal{S}_n\}$ and $\{D_{1/n}S_n\}$ are exponentially equivalent in the sense of Definition \ref{d.exp-equiv}. Therefore, Theorem \ref{t.exp-equiv} implies that, when  $\{D_{1/n}S_n\}$ satisfies an LDP, $\{D_{1/n}\mathcal{S}_n\}$ will satisfy an LDP with the same rate function. Thus, we may compare the application of Theorem \ref{t.main} to the Heisenberg group to \cite[Example 1]{BaldiCaramellino1999}.
\end{remark}

It will also be useful to recall that we say a random variable $X$ is \emph{sub-exponential} with parameters $\nu^2$ and $\alpha$ if
\[ \E[e^{\lambda X}]\leqslant e^{\nu^2\lambda^2/2}, \qquad \text{ for any } |\lambda|<\frac{1}{\alpha}. \]
We will write $X\in SE(\nu^2,\alpha)$. Similarly, a random vector is called \emph{sub-exponential} if it has all sub-exponential marginals. 
For Carnot groups of step 2 as in Example \ref{ex.step2}, the group operation is given by a quadratic polynomials. Thus, if the distribution of the $X_i$ is sub-Gaussian, we will observe a random walk with sub-exponential distributions. This will allow us to apply standard concentration results concerning quadratic forms and sub-exponential distributions in general. For convenience we record the following basic result.

\begin{lemma}\label{l.subexp}
If $Z_i$ are sub-exponential random variables  with parameters $\nu_i^2$ and $\alpha_i$, then the random variable $Z:= \sum_{i=1}^n Z_i$ is also sub-exponential with parameters $\alpha = \max_{1\leqslant i\leqslant n} \alpha_i$ and
	\begin{enumerate}
		\item[(i)] $\nu^2=\sum_{i=1}^n \nu_i^2$ if the $Z_i$s are independent.
		\item[(ii)] $\nu^2=\left(\sum_{i=1}^n \nu_i\right)^2$ if the $Z_i$s are dependent.
	\end{enumerate}
\end{lemma}

\begin{proof}
The first statement is obvious since $X$ and $Y$ independent implies that
\[
\E[e^{\lambda(X+Y)}]=\E[e^{\lambda X}]\E[e^{\lambda Y}].
\]
The second statement follows from an application of H\"older's inequality
\[
	\E[e^{\lambda(X+Y)}]\leqslant(\E[e^{p\lambda X}])^{1/p}\left(\E\left[e^{p\lambda Y/(p-1)}\right]\right)^{(p-1)/p}\leqslant e^{\nu_X^2p\lambda^2/2}e^{\nu_Y^2p\lambda^2/2(p-1)}
\]
followed by optimization over $p$.
\end{proof}

\begin{prop}\label{p.exp-approx}
Suppose $G$ is step 2 and that $\left\{ X_{i} \right\}_{i=1}^{\infty}$ are $\mathcal{H}$-valued i.i.d.~sub-Gaussian random vectors with mean $0$. Let $Y_n^m$ be as in \eqref{eq-Y-m} and $\Psi_m$ as given in \eqref{e.Psim}. Then for all $\delta>0$
\[
\lim_{m\to\infty} \limsup_{n\to\infty}\frac{1}{n} \log
		\p(\rho_{cc}(\Psi_m(Y_n^m), D_{1/n}S_n))>\delta) = -\infty.
\]
\end{prop}

\begin{proof}
Fix $\delta>0$. Recall that by the left invariance of the distance
\[
\rho_{cc}(\Psi_m(Y_n^m),D_{\frac{1}{n}}S_n) =
		\rho_{cc}(e,(\Psi_m(Y_n^m))^{-1}D_{\frac{1}{n}}S_n)
		= d_{cc}((\Psi_m(Y_n^m))^{-1}D_{\frac{1}{n}}S_n).
\]

Consider first the case $d_{2}=\mathrm{dim}[\mathcal{H},\mathcal{H}]=1$, which corresponds to the Heisenberg group, as in Example~\ref{ex.Heisenberg}.
Then by \eqref{e.2.5}
\[
D_{\frac{1}{n}} S_{n} = e^{X_{1}}\star\cdots\star e^{X_n} = \left(\frac{1}{n}\sum_{k=1}^nX_{k},
\frac{1}{n^2}\sum_{1\leqslant i<j\leqslant n} X_{i}^{T}A_{d_{1}+1} X_{j}\right)
\]
and
\begin{align*}
\Psi_m(Y_n^m) &=e^{Y_n^{m,1}}\star\cdots\star e^{Y_n^{m,m}}
\\
&= \left(\frac{1}{n}\sum_{k=1}^nX_{k}, \frac{1}{n^2} \sum_{k=1}^m \sum_{\ell=k+1}^m
	\sum_{i=n_{k-1}+1}^{n_{k}} \sum_{j=n_{\ell-1}+1}^{n_\ell} X_{i}^{T}A_{d_{1}+1}X_j\right). \end{align*}
Thus
\begin{equation*}
	(\Psi_m(Y_n^m))^{-1}\star D_{\frac{1}{n}}S_n
	=\left(0, \frac{1}{n^2}\sum_{k=1}^m \sum_{n_{k-1}<i<j\leqslant n_{k}}
		X_{i}^{T}A_{d_{1}+1} X_{j}\right).
	\end{equation*}
So in light of Proposition \ref{p.5.1.4} it is sufficient to show that
\begin{align*}
\lim_{m\to\infty} \limsup_{n\to\infty}\frac{1}{n} \log \p\left(\frac{1}{n^2}
	\left|\sum_{k=1}^m \sum_{n_{k-1}<i<j\leqslant n_{k}}
		X_{i}^{T}A_{d_{1}+1} X_{j}\right|>\delta\right)
		= -\infty.
\end{align*}

Now, by the comparison lemma \cite[Lemma 6.2.3]{VershyninBook2018}, for independent mean zero sub-Gaussian random vectors $X$ and $X'$ in $\mathbb{R}^{d_{1}}$ with $\Vert X  \Vert_{\psi_{2}} ,\Vert X'  \Vert_{\psi_{2}}\leqslant K$, and any $\lambda \in \mathbb{R}$ and  $d_{1}\times d_{1}$ matrix $A$
\[
\mathbb{E}[\exp\left( \lambda X^{T} A X' \right)]
	\leqslant \mathbb{E}[\exp\left( C_{1}K^{2}\lambda G^{T} A G' \right)]
\]
where $G$ and $G'$ are independent $\mathcal{N}\left( 0, I_{d_{1}} \right)$ random
	vectors and $C_{1}$ is an absolute constant. (Here, $\|\cdot\|_{\psi_{2}}$ is the sub-Gaussian norm, which is necessarily finite when $X$ is sub-Gaussian.)
Furthermore, by \cite[Lemma 6.2.2]{VershyninBook2018} there are absolute constants $C_{2}$ and $c$
\[
	\mathbb{E}[\exp\left( C_{1}K^{2} \lambda G^{T} A G' \right)]  \leqslant \exp\left( C_{2}C_{1} ^2 K^{4} \lambda^2 \Vert A \Vert^2_{HS}\right)
\]
for all $\lambda$  satisfying $\vert \lambda \vert \leqslant c/ C_{1}K^{2}  \Vert A \Vert_{op}$.

We denote by $C$ and $c$ constants that do not depend on the structure of the group or distribution, but might change from one bound to another. Putting the above together we see that for $X$ and $X'$ as given
\[
\mathbb{E}[\exp\left( \lambda X^{T} A X' \right)
	\leqslant  \exp\left( C K^{4} \Vert A \Vert_{HS}^{2}\lambda^{2} \right)],
	\quad \text{ for any } \vert \lambda \vert \leqslant \frac{c}{K^{2}  \Vert A \Vert_{op}}.
\]
This means that $X^{T} A X'$ is sub-exponential with parameters
$\nu^{2}:=2C K^{4} \Vert A \Vert_{HS}^{2}$ and $\alpha:=\frac{K^{2} \Vert A \Vert_{op}}{c}$.

In particular, the above is true for $X=X_i$ and $X'=X_j$ for any $i\ne j$ with  $A=A_{d_{1}+1}$. Now  for any $\ell\geqslant 2$ consider
\begin{align*}
	Z &:= \sum_{1\leqslant i<j\leqslant \ell} X_{i}^{T} A X_{j}
	= \sum_{a=3}^{2\ell-1}\sum_{\substack{1\leqslant i<j\leqslant\ell\\i+j=a}} X_{i}^{T} A X_{j}
	=: \sum_{a=3}^{2\ell-1} Z(a),
\end{align*}
where the second double summation is the same as summing along the \emph{antidiagonals} of the array $(X_{i}^{T} A X_{j})_{1\leqslant i<j\leqslant\ell}$. We will apply Lemma \ref{l.subexp} to recognize this sum as having a sub-exponential distribution and to bound its parameters.
Note in particular that each $Z(a)$ is a sum of \emph{independent} sub-exponential random variables $X_{i}^{T} A X_{j}$ all with common parameters $\nu_{ij}^2=\nu^2$ and $\alpha_{ij}=\alpha$, where $\nu^2$ and $\alpha$ are as given above. Thus, $Z(a)$ is sub-exponential with parameters $\alpha_a=\alpha$ and
\[
\nu_a^2 = \sum_{\substack{1\leqslant i<j\leqslant\ell\\i+j=a}} \nu^2 \leqslant \frac{\ell}{2}\nu^2. \]
Therefore, as a sum of \emph{dependent} sub-exponential random variables, $Z=\sum_{a=3}^{2\ell-1} Z(a)$ is sub-exponential with parameters $\alpha_Z=K^2\|A\|_{op}/c$, and we may make the following (rough) estimate
\[
\nu_Z^2 = \left(\sum_{a=3}^{2\ell-1}\nu_a\right)^2
	\leqslant (2\ell-3)^2\frac{\ell}{2}\nu^2
	\leqslant \ell^3C K^{4} \Vert A \Vert_{HS}^{2}.
\]
Applying this now to
\[
Z_{k}:=\sum_{n_{k-1}<i<j\leqslant n_{k}} X_{i}^{T} A X_{j},
\]
for any $k\in [m]$, we have $Z_{k} \in SE(d^3C K^{4} \Vert A \Vert_{HS}^{2},K^2\|A\|_{op}/c)$, where again $d=\lfloor n/m\rfloor$ (and $n_{k}-n_{k-1}=d$ for $k\in[m-1]$ and $n_m-n_{m-1}=d+r$ for some $r \in[m-1]$). Since $Z_{k}$ and $Z_{k'}$ for $k\not= k'$ are sums over non-overlapping subsets of indices, we have that the $Z_{k}$'s are also independent, so
\[
\sum_{k=1}^{m}Z_{k}\in SE\left(md^{3}CK^{4} \Vert A \Vert_{HS}^{2},\frac{K^{2} \Vert A \Vert_{op}}{c}\right).
\]

Thus by \eqref{e.standard} for any $\delta>0$ and $0< \lambda\leqslant \frac{c}{K^{2}  \Vert A \Vert_{op}}$
\begin{align*}
& \mathbb{P}\left(\frac{1}{n^{2}} \left\vert \sum_{k=1}^{m}Z_{k}\right\vert > \delta \right)
\leqslant
2e^{- \lambda \delta n^{2}}e^{ md^{3}CK^{4} \Vert A \Vert_{HS}^{2}\lambda^{2}}.
\end{align*}
In particular, this is true for any $\lambda= 1/d=1/\lfloor n/m\rfloor$ for sufficiently large $n$ and thus
\begin{align*}
\lim_{m\to\infty} \limsup_{n\to\infty}\frac{1}{n} \log  \mathbb{P}&\left(\frac{1}{n^{2}} \left\vert \sum_{k=1}^{m}\sum_{n_{k-1}<i<j\leqslant n_{k}} X_{i}^{T} A X_{j}\right\vert > \delta \right)
	\\
	&= \lim_{m\to\infty} \limsup_{n\to\infty} \frac{1}{n} \log  \mathbb{P}\left(\frac{1}{n^{2}} \left\vert \sum_{k=1}^{m}Z_{k}\right\vert > \delta \right) \\
	&\leqslant \lim_{m\to\infty} \limsup_{n\to\infty}\frac{1}{n} \log \left(2e^{- m\delta n}
		e^{ mdCK^{4} \Vert A \Vert_{HS}^{2}}\right) \\
	&=\lim_{m\to\infty}\limsup_{n\to\infty}
		\left( - \delta m+CK^{4} \Vert A \Vert_{HS}^{2}\right)
		=-\infty.
\end{align*}

This essentially completes the proof, since in the case that ${d_{2}}= \mathrm{dim}[\mathcal{H},\mathcal{H}]>1$ and  $(\Psi_m(Y_n^m))^{-1}\star D_{\frac{1}{n}}S_n = (0,Z_n^1,\ldots,Z_n^{d_{2}})$ with
\[
Z_n^\ell = \frac{1}{n^2}\sum_{k=1}^m \sum_{n_{k-1}< i<j\leqslant n_{k}} X_{i}^{T} A_{d_{1}+\ell} X_{j},
\]
we have that
\begin{align*}
\p\left(\left\|(Z_n^1,\ldots,Z_n^{d_{2}})\right\|_{\R^{d_{2}}}>\delta\right)
	&= \p\left(\sum_{\ell=1}^{d_{2}} (Z_n^\ell)^2>\delta^2\right) \\
	&\leqslant \sum_{\ell=1}^{d_{2}} \p\left( (Z_n^\ell)^2>\frac{\delta^2}{{d_{2}}}\right)
	= \sum_{\ell=1}^{d_{2}} \p\left( |Z_n^\ell|>\frac{\delta}{\sqrt{{d_{2}}}}\right)
\end{align*}
and the result follows by applying the previous estimates to each term.
\end{proof}
Note that Appendix~\ref{a.RSymplF} describes properties of random quadratic forms. If $G$ is of step $3$ or higher, we need to rely on concentration inequalities for polynomials of random vectors of higher order which are not easily available.

\subsubsection{Higher step groups with bounded distributions}

Suppose $G$ is a homogeneous Carnot group of step $r$ as described in Section~\ref{s.carnot}.
In particular, as before we identify both the Lie group $G$ and its Lie algebra $\mathfrak{g}$ with $\mathbb{R}^{N}$.

First we will need the next simple lemma for our estimates.
\begin{lemma}\label{l.Jl} Recall the notation \eqref{eq-step-cord}. For $\ell=2,\ldots,r$
	\[ (D_{\frac{1}{n}}S_n-\Psi_m(Y_n^m))^{(\ell)}
		= \frac{1}{n^\ell}\sum_{i=(i_{1},\ldots,i_\ell)\in \mathcal{J}'_\ell} c_i \operatorname{ad}_{X_{i_{1}}}\cdots \operatorname{ad}_{X_{i_{\ell-1}}}
			X_{i_\ell} \]
	for some coefficients $|c_i|<1$, where $\mathcal{J}'_\ell$ is some strict subset of $\{1,\ldots,n\}^\ell$ satisfying $\#\mathcal{J}'_\ell\leqslant C \frac{n^\ell}{m}$ where $C$ is a constant that only depends on $\ell$.
\end{lemma}

\begin{proof}
	As with Lemma \ref{l.Il}, the form of $(D_{\frac{1}{n}}S_n-\Psi_m(Y_n^m))^{(\ell)}$ follows from the Baker-Campbell-Hausdorff-Dynkin formula \eqref{e.BCDH} along with the definition of the dilation. So we only need to  prove the bound on $\#\mathcal{J}'_\ell$. However, we may note that
	
\[ 
\mathcal{J}'_\ell\subseteq\{i=(i_{1},\ldots,i_\ell): i_{1},\ldots,i_\ell\in\{n_{k-1}+1,\ldots, n_{k}\} \text{ for some } k \in[m]\}, 
\] 
and so $\#\mathcal{J}'_\ell\leqslant m\cdot C(n/m)^\ell \leqslant C \frac{n^\ell}{m} $.
\end{proof}

Clearly the estimate in the lemma above on the number of terms appearing in the sum is rough, but it is sufficient for our purposes. We are now able to prove that under these conditions $\{\Psi_m(Y_n^m)\}$ are exponentially good approximations to $\{D_{\frac{1}{n}}S_n\}$.

\begin{prop}\label{p.exp-approx-higher}
Suppose that $\left\{ X_{i} \right\}_{i=1}^{\infty}$ are i.i.d.~mean 0 bounded random vectors in $\mathcal{H}$. Then for all $\delta>0$
\[
\lim_{m\to\infty} \limsup_{n\to\infty}\frac{1}{n} \log \p(\rho_{cc}(\Psi_m(Y_n^m),D_{\frac{1}{n}}S_n)>\delta) = -\infty.
\]
\end{prop}

\begin{proof}
Suppose $|X_i|$, $i\ge1$ are a.s.~bounded by some constant $M>1$. First, note that by Lemma \ref{l.Il}
\begin{align*}
		\|D_{\frac{1}{n}}S_n\|_{\mathbb{R}^N} &\leqslant \sum_{\ell=1}^r \frac{1}{n^\ell} \|S_n^{(\ell)}\|_{\R^{d_\ell}}
			\leqslant \sum_{\ell=1}^r \frac{1}{n^\ell}\sum_{i\in \mathcal{J}_\ell} |c_i| \|\operatorname{ad}_{X_{i_{1}}}\cdots \operatorname{ad}_{X_{i_{\ell-1}}}
			X_{i_\ell}\|_{\R^{d_\ell}} \\
			&\leqslant \sum_{\ell=1}^r \frac{1}{n^\ell}  n^\ell (2M)^\ell
			\leqslant r(2M)^r.
\end{align*}
Thus for all $n$, $D_{\frac{1}{n}}S_n$ is in some compact subset of $\mathbb{R}^N$ with diameter depending only on $M$ and $r$, and similarly for $\Psi_m(Y_n^m)$. Thus, by Proposition \ref{p.NormEquiv}, it suffices to prove that
\[
\lim_{m\to\infty} \limsup_{n\to\infty}\frac{1}{n} \log \p(\|D_{\frac{1}{n}}S_n-\Psi_m(Y_n^m)\|_{\mathbb{R}^N}>\delta) = -\infty,
\]
or rather that, for each $\ell=2,\ldots, r$,
\[
\lim_{m\to\infty} \limsup_{n\to\infty}\frac{1}{n} \log \p(\|(D_{\frac{1}{n}}S_n-\Psi_m(Y_n^m))^{(\ell)}\|_{\R^{d_\ell}}>\delta) = -\infty.
\]
	So fix $\ell\in\{2,\ldots,r\}$. By Lemma \ref{l.Jl} we have that
\begin{align*}
	\p(\|(D_{\frac{1}{n}}S_n-\Psi_m(Y_n^m))^{(\ell)}\|_{\R^{d_\ell}}>\delta)
	 &= \p\left(\frac{1}{n^\ell}\left\|\sum_{i\in \mathcal{J}'_\ell} c_i \operatorname{ad}_{X_{i_{1}}}\cdots \operatorname{ad}_{X_{i_{\ell-1}}} X_{i_\ell}\right\|_{\R^{d_\ell}}>\delta\right) \\
	 &\leqslant \p\left(\frac{1}{n^\ell} \sum_{i\in \mathcal{J}'_\ell}
	 	\left\| \operatorname{ad}_{X_{i_{1}}}\cdots \operatorname{ad}_{X_{i_{\ell-1}}}
			X_{i_\ell}\right\|_{\R^{d_\ell}}>\delta\right) \\
	 &\leqslant \sum_{i\in \mathcal{J}'_\ell} \p\left(\frac{1}{n^\ell} \left\| \operatorname{ad}_{X_{i_{1}}}\cdots \operatorname{ad}_{X_{i_{\ell-1}}}
			X_{i_\ell}\right\|_{\R^{d_\ell}}>\frac{\delta m}{Cn^\ell}\right)
\end{align*}
	
Since the distribution of the $X_i$s is bounded in $\h$, there exists an $M$ such that $\p(|X_i|_\h\geqslant M)=0$. Since $\Vert \operatorname{ad}_{X}Y \Vert \leqslant 2 |X|_\h|Y|_\h  $, this implies that for any multi-index $i\in \mathcal{J}'_\ell$
\[
\p\left( \|\operatorname{ad}_{X_{i_{1}}}\cdots \operatorname{ad}_{X_{i_{\ell-1}}}
			X_{i_\ell}\|_{\R^{d_\ell}}\geqslant \frac{M^\ell}{2^\ell}\right) =0.
\]
Thus $\| \operatorname{ad}_{X_{i_{1}}}\cdots \operatorname{ad}_{X_{i_{\ell-1}}}
			X_{i_\ell}\|_{\R^{d_\ell}}$ satisfies \eqref{e.standard1} with $C(\lambda) = e^{M^\ell\lambda}$ for any $\lambda>0$ for each $i\in \mathcal{J}'_\ell$, and it follows by
\eqref{e.standard} that
\begin{align*}
	\sum_{i\in \mathcal{J}'_\ell} \p\left( \| \operatorname{ad}_{X_{i_{1}}}\cdots \operatorname{ad}_{X_{i_{\ell-1}}}
			X_{i_\ell}\|_{\R^{d_\ell}}>\frac{\delta m}{C}\right)
		&\leqslant  C\frac{n^\ell}{m} \cdot 2 \exp\left(-\lambda \frac{\delta m}{C}\right)\exp(M^\ell\lambda).
\end{align*}
In particular, this is true for $\lambda =n$ for any $n$. Thus,
\[
\frac{1}{n}\log \p(\|(D_{\frac{1}{n}}S_n-\Psi_m(Y_n^m))^{(\ell)}\|_{\R^{d_\ell}}>\delta)
		\leqslant \frac{1}{n}\log \left(2C \,\frac{n^\ell}{m}
			\exp\left(- n\frac{\delta m}{C}\right)\exp(M^\ell n)\right).
\]
Putting this all together gives
\begin{align*}
	\limsup_{n\to\infty} \frac{1}{n}\log \p(\|(D_{\frac{1}{n}}S_n-\Psi_m(Y_n^m))^{(\ell)}\|_{\R^{d_\ell}}>\delta)
	&\leqslant \limsup_{n\to\infty} \frac{1}{n}\left( - n\frac{\delta m}{C} + M^\ell n\right) \\
	&= - \frac{\delta m}{C} + M^\ell
\end{align*}
and taking $m\to\infty$ completes the proof.
\end{proof}
\subsubsection{Higher step groups with a Gaussian distribution} Again, we let $G$ be a homogeneous Carnot group of step $r$ and identify it and its Lie algebra $\mathfrak{g}$ with $\mathbb{R}^{N}$.
We'll now prove that, when $\{X_i\}_{i=1}^n$ are i.i.d.~standard Gaussian random variables on $\mathcal{H}$,  $\{\drw\}$ form exponentially good approximations of $\{\rw\}$.

\begin{prop} \label{prop:mainprop}
	Suppose that $\left\{ X_{i} \right\}_{i=1}^{\infty}$ are i.i.d.~$\mathcal{N}(0,\mathrm{Id}_\h)$. Then for all $\delta > 0$
	\begin{align*}
		\lim\limits_{m\to \infty}\limsup\limits_{n\to\infty} \frac{1}{n}\log\p\left(\rho_{\mathrm{cc}}\left(\rw,\drw\right)  \geqslant \delta\right) = -\infty.
	\end{align*}
\end{prop}
Before explaining the proof of Proposition \ref{prop:mainprop} we first introduce some  specific conventions and notation. If $E$ is an inner product space, $\|\cdot\|_E$ will always stand for the Hilbertian norm on $E$. When the space is clear from the context we will abbreviate to $\|\cdot\|$. The notations $C, C^{\prime},C_1,C_2,\ldots$ will stand for constants which may depend on $G$, but not on $n$ and $m$.  

We also introduce the following projection operators which somewhat streamline and refine the projections induced by the stratification, as in \eqref{e.Stratification}.
\begin{notation}\label{n.Projections}
	For any step $r$ homogeneous Carnot group $G \cong \R^N$ and element $x = (x^{(1)},\dots,x^{(r)}) \in G$, if $\ell \in [r],$ we denote by $\Pi_{\ell}: \mathbb{R}^{N} \longrightarrow \mathbb{R}^{d_{\ell}}$ the projection onto the $\ell^{\mathrm{th}}$ step of $G$
	\begin{align*}
		\Pi_{\ell}x:= x^{(\ell)}=\left(  x_{d_{1}+\cdots+d_{\ell-1}+1}, \dots ,  x_{d_{1}+\cdots+d_{\ell}}\right).
	\end{align*}
	Moreover, for any $j \in [d_\ell]$,  $\Pi^j_{\ell}: \mathbb{R}^{N} \longrightarrow \mathbb{R}$ will be the projection onto the $j^{\mathrm{th}}$ coordinate of $\mathrm{Image}(\Pi_\ell)$,
	\begin{align*}
		\Pi_{\ell}^{j}x:=x_{d_{1}+\cdots+d_{\ell-1}+j}.
	\end{align*}
\end{notation}

With the above conventions, our proof of Proposition \ref{prop:mainprop} relies on the following result, specializing to separate Euclidean approximations on each step of $G$.
\begin{prop} \label{prop:stepapprox}
	Let $\delta > 0$ and $\ell \in [r]$. It holds that
	\begin{align*}
		\lim\limits_{m\to \infty}\limsup\limits_{n\to\infty} \frac{1}{n}\log\p\left(\left\|\Pi_\ell\left((\drw)^{-1}\star\rw \right)\right\|  \geqslant \delta\right) = -\infty.
	\end{align*}
	
\end{prop}

Given Proposition \ref{prop:stepapprox}, we may now prove Proposition \ref{prop:mainprop}.

\begin{proof}[Proof of Proposition \ref{prop:mainprop}]
	First, the Carnot-Carath\'{e}odory metric is left-invariant, thus
	\[
	\rho_{\mathrm{cc}}\left(\rw,\drw\right) = \rho_{\mathrm{cc}}\left((\drw)^{-1}\star\rw,0\right).
	\]
	Second, since all homogeneous norms on $G$ are equivalent by Proposition \ref{p.5.1.4}, we have
	\[
	\rho_{\mathrm{cc}}\left((\drw)^{-1}\star\rw,0\right) \leqslant \left(C\sum\limits_{\ell = 1}^r \left\|\Pi_\ell\left((\drw)^{-1}\star\rw \right)\right\|^{\frac{2r!}{\ell}}\right)^{\frac{1}{2r!}}.
	\]
	So,
	\begin{align*}
		\p\left(\rho_{\mathrm{cc}}\left(\rw,\drw\right)\geqslant \delta\right) &\leqslant  \p\left(\sum\limits_{\ell=1}^r\left\|\Pi_\ell\left((\drw)^{-1}\star\rw \right)\right\|^{\frac{2r!}{\ell}}\geqslant \delta^{2r!}/C\right)
		\\
		& \leqslant \sum\limits_{\ell=1}^r\p\left(\left\|\Pi_\ell\left((\drw)^{-1}\star\rw \right)\right\|\geqslant \frac{\delta^\ell}{(Cr)^{\frac{\ell}{2r!}}}\right).
	\end{align*}
	The result follows by applying Proposition~\ref{prop:stepapprox} to each summand separately with $\delta$ replaced by $\frac{\delta^\ell}{(Cr)^{\frac{\ell}{2r!}}}$, and proceeding in a similar fashion to the proofs of Propositions \ref{p.exp-approx} and \ref{p.exp-approx-higher}.
\end{proof}

So we henceforth focus our attention on the proof of Proposition \ref{prop:stepapprox}. This requires understanding the group operation on each step of $G$. In light of \eqref{e.SymplForm} we introduce the following polynomial functionals in Carnot groups.

\begin{definition}[Polynomials in Carnot groups]
	Let $G$ be a Carnot group of step $r$.
	\begin{itemize}
		
		\item Let $\beta,\beta^{\prime} \in [r]$. A function $\omega:\R^N\times \R^N \to \R$  is called a \emph{$(\beta, \beta^{\prime})$-symplectic form} if there exist $j_\beta \in [d_\beta]$ and $j_{\beta^{\prime}} \in [d_{\beta^{\prime}}]$ such that 
		\begin{equation} \label{eq:symlectic}
			\omega(x,y) = (\Pi^{j_\beta}_{\beta}x) (\Pi^{j_{\beta^{\prime}}}_{\beta^{\prime}}y) - (\Pi^{j_\beta}_{\beta}y) (\Pi^{j_{\beta^{\prime}}}_{\beta^{\prime}}x).
		\end{equation}
		\item Let $\alpha \in [r]$. A function $P:\R^N\times \R^N \to \R$ is called an \emph{$\alpha$-monomial} if it is of the form
		\begin{equation} \label{eq:monomial}
			P(x,y) := \prod\limits^{r'}_{i=1} \Pi^{j_i}_{\ell_i}z_i,
		\end{equation}
		for some $r' \leqslant r$, and where for each $i$, $z_i \in \{x,y\}$, $\ell_i \in [r]$, and $j_i \in [d_{\ell_i}]$. Moreover, we have the degree constraint $\sum \ell_i = \alpha$.
		\item A pair $(\omega,P)$ of a $(\beta,\beta')$-symplectic form and an $\alpha$-monomial is called an \emph{$\ell$-homogeneous pair} if 
		\begin{equation} \label{eq:degsmatch}
			\ell = \alpha + \beta +\beta'.
		\end{equation}
	\end{itemize}
\end{definition}

With these definitions, the main observation is that the group operation can be decomposed into a linear part and a homogeneous part. Formally, for every $\ell \in [r]$, $j \in [d_\ell]$, and $x,y \in \R^N$, we rewrite \eqref{e.3.3} and \eqref{e.SymplForm} as,
\begin{equation} \label{eq:groupopdecomp}
	\Pi^j_\ell (y^{-1}\star x) = \Pi^j_\ell(x-y) + \sum\limits_{k=1}^{K_\ell^j}\omega^k(x,-y) P^k(x,-y).
\end{equation}
Here $K_\ell^j$ is some fixed number and for every $k \in [K_\ell^j]$, $(\omega^k, P^k)$ is an $\ell$-homogeneous pair. 

We shall handle the linear part and the polynomial part in \eqref{eq:groupopdecomp} separately. Towards this we require the following concentration inequalities for polynomials in Gaussian variables, which follow from hypercontractivity. The proof of which can be found in \cite[Chapters 5 and 6]{JansonBook1997}.
\begin{lemma}[Theorem 5.10 and Theorem 6.7 with Remark 6.8 in \cite{JansonBook1997}]\label{lem:gausspolyconc}
	Let $X_1,\dots,X_n$ be i.i.d. standard Gaussian random variables in some Euclidean space $E$ and let $Q:E \to \R$ be a degree-$p$ polynomial. Then
	\begin{itemize}
		\item[(1).] For any $q \geqslant 1$,
		\[
		\left(\E\left[|Q(X_1,\dots,X_n)|^q\right]\right)^{\frac{1}{q}} \leqslant C_{q,p}\sqrt{\E\left[|Q(X_1,\dots,X_n)|^2\right]},
		\]
		where $C_{q,p}>0$ depends only on $p$ and $q$.
		\item[(2).]  For and $\delta > 0$,
		\[
		\p\left(|Q(X_1,\dots,X_n)| \geqslant \delta\right)\leqslant \exp\left(-C_p\left(\frac{\delta}{\sqrt{\E\left[|Q(X_1,\dots,X_n)|^2\right]}}\right)^{\frac{2}{p}}\right),
		\]
		where $C_p>0$ depends only on $p$.
	\end{itemize}
\end{lemma}
{\bf Bounding the linear part:} Our task of bounding the elements appearing in \eqref{eq:groupopdecomp} starts with the linear part, for which we will need the following second moment estimates.
\begin{lemma} \label{lem:variances}
	Let $\ell \in [r]$. It holds that
	\begin{enumerate}
		\item 
		\[
		\E\left[\|\Pi_\ell \rw\|^2\right], \E\left[\|\Pi_\ell \drw\|^2\right] \leqslant \frac{C}{n^{\ell}}.
		\]
		\item 
		\[
		\E\left[\|\Pi_\ell\left(\rw - \drw\right)\|^2\right] \leqslant \frac{C}{mn^{\ell}}.
		\]
	\end{enumerate}
\end{lemma}
\begin{proof}
	For the first claim, by the Baker-Campbell-Hausdorff-Dynkin formula, \eqref{e.BCDH}, we have
	\[
	\Pi_\ell \rw = \frac{1}{n^\ell}\sum\limits_{I\in [n]^\ell} c_I\ad_{X_{I_1}}\dots\ad_{X_{I_{\ell-1}}}X_{I_\ell},
	\]
	where $|c_I| \leqslant 1$ and $I=(I_1,\dots I_\ell)$. For $I \in [n]^\ell$ we abbreviate
	\[
	X_I := \ad_{X_{I_1}}\dots\ad_{X_{I_{\ell-1}}}X_{I_\ell}, 
	\]
	so that,
	\begin{equation} \label{eq:2ndmoment}
		\E\left[\|\Pi_\ell \rw\|^2\right] \leqslant \frac{1}{n^{2\ell}}\sum\limits_{I,J \in [n]^\ell}\left|\E\left[X_I\cdot X_J\right]\right|.
	\end{equation}
	Now, for $I,J \in [n]^\ell$, consider the multiset, 
	$$I \cup J := \{I_1,\dots,I_\ell,J_1,\dots, J_\ell\},$$
	and suppose that there exists $k \in I \cup J$ which appears an odd number of times. In this case, by the symmetry of standard Gaussian random variables, and the bi-linearity of the Lie brackets
	\begin{equation*}
		\E\left[X_I\cdot X_J\right] = -\E\left[X_I\cdot X_J\right] \implies \E\left[X_I\cdot X_J\right] = 0.
	\end{equation*} 
	Thus, every non-zero summand in \eqref{eq:2ndmoment} must satisfy that every element in $I\cup J$ appears at least twice. Since there are at most $Cn^\ell$ such pairs we conclude
	\[
	\E\left[\|\Pi_\ell \rw\|^2\right] \leqslant \frac{1}{n^{2\ell}}Cn^\ell = \frac{C}{n^\ell}.
	\]
	In the above we have used the fact that the distribution of $X_I \cdot X_J$ only depends on the number of identical elements and their positions. Since there is a finite, in $\ell$, number of such combinations, we have that
	\[
	\E\left[X_I\cdot X_J\right] \leqslant C,
	\]
	where $C>0$ depends only on $\ell$. This concludes the bound of $\E\left[\|\Pi_\ell \rw\|^2\right]$. The proof for $\E\left[\|\Pi_\ell \drw\|^2\right]$ is completely identical, with less relevant pairs in \eqref{eq:2ndmoment}, and we omit it.
	
	For the second part of the claim, we apply the Baker-Campbell-Hausdorff-Dynkin formula to $\Pi_\ell\left(\rw - \drw\right)$, as in Lemma \ref{l.Jl},
	\[
	\Pi_\ell\left(\rw - \drw\right)=\frac{1}{n^\ell}\sum\limits_{I\in \mathcal{I}_\ell} c_I\ad_{X_{I_1}}\cdots\ad_{X_{I_{\ell-1}}}X_{I_\ell},
	\]
	where again $|c_I| \leqslant 1$, and $\mathcal{I}_\ell \subset [n]^\ell$ is such that for every $i \in[n]$, 
	\[
	\left|\{I_{\ell -1}: I \in \mathcal{I}_\ell \text{ and } I_\ell = i\} \right| \leqslant \frac{n}{m}.
	\]
	This property says that, once the last element is chosen, there are at most $\frac{n}{m}$ different choices for the next element, and in particular $|\mathcal{I}_\ell| \leqslant \frac{n^\ell}{m}$.
	As before, it will be enough to count the number of pairs $I, J \in \mathcal{I}_\ell$ such that every element in $I\cup J$ appears an even number of times. Thus, we need to choose at most $n^\ell$ elements. There are $n$ choices for $I_\ell$ and once this element is chosen there are only $\frac{n}{m}$ for $I_{\ell-1}$ and necessarily $I_\ell \neq I_{\ell-1}$. The rest of the elements have at most $n$ choices, and there is a finite, in $\ell$, number of ways to arrange them. Altogether, there are at most $C\frac{n^\ell}{m}$ pairs in $\mathcal{I}_\ell$, for which,
	\[
	\E\left[X_I\cdot X_J\right] \not= 0.
	\]
	The conclusion of the proof is identical to the previous part.
\end{proof}
We can now use Lemma \ref{lem:variances} along with appropriate Gaussian concentration results to bound the linear part in \eqref{eq:groupopdecomp}.
\begin{lemma} \label{lem:differencebound}
	Let $\delta > 0$ and $\ell \in [r]$. It holds that,
	\begin{align*}
		\p\left(\left\|\Pi_\ell\left(\rw - \drw\right)\right\| \geqslant \delta\right) \leqslant \exp\left(-C_1(\delta^2 m)^{\frac{1}{\ell}}n\right),
	\end{align*}
	for a constant $C_1 > 0$ that may depend on $\ell$ and $G$, but not on $n$ and $m$.
\end{lemma}
\begin{proof}
	Observe that, by the Baker-Campbell-Hausdorff-Dynkin formula, \eqref{e.BCDH}, every entry of  $\Pi_\ell\left(\rw - \drw\right)$ is a degree $\ell$ polynomial in the Gaussian random variables $X_1,\dots, X_n$. Thus, by Lemma \ref{lem:gausspolyconc} we have
	\begin{align*}
		\p\left(\|\Pi_\ell\left(\rw - \drw\right)\| \geqslant \delta \right) &\leqslant \exp\left(-C\left(\frac{\delta}{\sqrt{\E\left[\|\Pi_\ell\left(\rw - \drw\right)\|^2\right]}}\right)^{\frac{2}{\ell}}\right)\\
		&\leqslant \exp\left(-C(\delta^2 m)^{\frac{1}{\ell}}n\right),
	\end{align*}
	where the second inequality is due to Lemma \ref{lem:variances} (2).
\end{proof}
{\bf Bounding the polynomial part:}
To bound the polynomial part in \eqref{eq:groupopdecomp}, we focus on a single summand. Thus, let $\omega$ be a $(\beta,\beta')$-symplectic form, as in \eqref{eq:symlectic}, and let $P$ be an $\alpha$-monomial, in the form of \eqref{eq:monomial}. Assume further that $(\omega, P)$ form an $\ell$-homogeneous pair so that, as in \eqref{eq:degsmatch}, $\alpha + \beta + \beta' = \ell.$

As before, we shall require a second moment estimate, this time for monomials.
\begin{lemma} \label{lem:hypercontractivity}
	Let $P$ be an $\alpha$-monomial, as in \eqref{eq:monomial},  
	\[
	P(X,Y) = \prod\limits_{i=1}^{r'}\Pi^{j_i}_{\ell_i} Z_i,
	\]
	with $\sum \ell_i = \alpha,$ and $r' \leqslant \alpha$. Then for Gaussian random variables $X_1,\dots, X_n$,
	\[
	\E\left[P(\rw, -\drw)^2\right] \leqslant \frac{C}{n^\alpha},
	\]
	for a constant $C>0$ which may depend on $\alpha$. 	Consequently, for any $\delta > 0$,
	\[
	\p\left(\left|P(\rw, -\drw)\right|\geqslant \delta\right) \leqslant \exp\left(-C\delta^{\frac{2}{\alpha}}n\right).
	\]
\end{lemma}
\begin{proof}
	By H\"older's inequality, 
	\[
	\E\left[P(\rw, -\drw)^2\right]\leqslant \left(\prod_{i=1}^{r'}\E\left[\|\Pi_{\ell_i} Z_i\|^{2r'}\right]\right)^{\frac{1}{r'}},
	\]
	where for each $i$, $Z_i \in \{\rw, -\Psi_n^m\}$. Now, by the Baker-Campbell-Hausdorff-Dynkin formula, \eqref{e.BCDH}, for each $i$, both $\Pi_\ell \rw$ and $\Pi_\ell\Psi_n^m$ are degree-$\ell$ polynomials in the Gaussian variables $X_1,\dots, X_n$. Applying Lemma \ref{lem:gausspolyconc} we obtain,
	\[
	\E\left[\|\Pi_{\ell_i} Z_i\|^{2r'}\right] \leqslant C'_\alpha\E\left[\|\Pi_{\ell_i} Z_i\|^2\right]^{r'} \leqslant \frac{C_\alpha}{n^{r'\ell_i}},
	\]
	where the second inequality is due to Lemma \ref{lem:variances} (1), and $C_\alpha,C'_\alpha > 0$ are constants which can be chosen to only depend on $\alpha$. Thus,
	\[
	\E\left[P(\rw, -\drw)^2\right] \leqslant C_\alpha^{\frac{1}{r'}}\left(\prod_{i=1}^{r'}\frac{1}{n^{r'\ell_i}}\right)^{\frac{1}{r'}} = C_\alpha^{\frac{1}{r'}}\frac{1}{n^{\alpha}}.
	\]
	Above we have used $\sum \ell_i = \alpha$. This relation also shows that $P(\rw, -\drw)$ is a degree-$\alpha$ polynomial in the Gaussian variables, and hence by Lemma \ref{lem:gausspolyconc},
	\begin{align*}
		\p\left(\left|P(\rw, -\drw)\right|\geqslant \delta\right) &\leqslant \exp\left(-C_\alpha''\left(\frac{\delta}{\sqrt{\E\left[P(\rw, -\drw)^2\right]}}\right)^{\frac{2}{\alpha}}\right)\\
		&\leqslant \exp(-C\delta^{\frac{2}{\alpha}}n).
	\end{align*}
\end{proof}
With Lemma \ref{lem:hypercontractivity} we can now bound the homogeneous part in \eqref{eq:groupopdecomp}.
\begin{lemma} \label{lem:homoegprob}
	Let $\omega$ and $P$ be as above with $\alpha + \beta + \beta' = \ell$. Then,
	\[
	\p\left(\left|\omega(D_{\frac{1}{n}}S_n, -\drw)P(\rw, -\drw)\right| \geqslant \delta\right) \leqslant \exp(-C(\delta^2m)^{\frac{1}{\ell}} n)
	\]
	for some $C > 0$, which can depend on $\alpha,\beta,\beta'$, and $G$, but not on $n$ and $m$.
\end{lemma}
\begin{proof}
	Let us write $S= \Pi_\beta^{j_\beta} \rw$, $S'=\Pi^{j'_\beta}_{\beta'} \rw$, and similarly for $\Psi, \Psi'$. So,
	\[
	\omega(D_{\frac{1}{n}}S_n, -\drw) = S\Psi' - S'\Psi = (S-\Psi)\Psi' + (\Psi'-S')\Psi.
	\]
	Let us fix $M,M'>0$ to be chosen later and apply the union bound,
	\begin{align*}
		\p&\left(\left|\omega(D_{\frac{1}{n}}S_n, -\drw)P(\rw, -\drw)\right| \geqslant \delta\right)\\
		&\leqslant \p\left(|S-\Psi||\Psi'P(\rw, -\drw)| \geqslant \frac{\delta}{2}\right) + \p\left(|S'-\Psi'||\Psi P(\rw, -\drw)| \geqslant \frac{\delta}{2}\right)\\
		&\leqslant \p\left(\|\Pi_\beta(\rw - \drw)\| \geqslant \frac{\delta}{2M}\right) + \p\left(\|\Pi_{\beta'}(\rw - \drw)\| \geqslant \frac{\delta}{2M'}\right) \\
		&\ \ \ \ +\p\left(|\Psi'P(\rw, -\drw)| \geqslant M\right) + \p\left(|\Psi P(\rw, -\drw)| \geqslant M'\right).
	\end{align*}
	By Lemma \ref{lem:differencebound}, 
	\begin{align} \label{eq:normbound}
		\p&\left(\|\Pi_\beta(\rw - \drw)\| \geqslant \frac{\delta}{2M}\right) + \p\left(\|\Pi_{\beta'}(\rw - \drw)\| \geqslant \frac{\delta}{2M'}\right)\nonumber\\
		&\leqslant \exp\left(-C\left(\frac{\delta^2}{M^2}m\right)^{\frac{1}{\beta}}n\right) + \exp\left(-C\left(\frac{\delta^2}{M'^2}m\right)^{\frac{1}{\beta'}}n\right).
	\end{align}
	The other terms will be bounded using Lemma \ref{lem:hypercontractivity}. Indeed, observe that $\Psi P(\rw, -\drw)$ is a monomial of degree $\alpha + \beta$ and that $\Psi' P(\rw, -\drw)$ is a monomial of degree $\alpha + \beta'$. Since, necessarily, $\alpha + \beta, \alpha + \beta' \leqslant \ell$, Lemma \ref{lem:hypercontractivity} gives,
	\begin{align} \label{eq:monomialbound}
		\p\left(|\Psi P(\rw, -\drw)| \geqslant M'\right)&\leqslant \exp\left(-CM'^{\frac{2}{\beta + \alpha}}n\right),\nonumber \\
		\p\left(|\Psi'P(\rw, -\drw)| \geqslant M\right)&\leqslant \exp\left(-CM^{\frac{2}{\beta' + \alpha}}n\right).
	\end{align}
	
We now choose $M = (\delta^2m)^{-\frac{\beta+\alpha}{2\ell}}$ and $M' = (\delta^2m)^{-\frac{\beta'+\alpha}{2\ell}}$. Plugging these choices into \eqref{eq:normbound} and \eqref{eq:monomialbound}, and recalling $\alpha + \beta + \beta' = \ell$, we get
	\[
	\p\left(\left|\omega(D_{\frac{1}{n}}S_n, -\drw)P(\rw, -\drw)\right| \geqslant \delta\right)\leqslant \exp(-C(\delta^2m)^{\frac{1}{\ell}} n).
	\]

\end{proof}
\textbf{Finishing the proof:}
With Lemma \ref{lem:differencebound} and Lemma \ref{lem:homoegprob} we can now prove Proposition \ref{prop:stepapprox}.
\begin{proof}[Proof of Proposition \ref{prop:stepapprox}]
	Applying the decomposition in \eqref{eq:groupopdecomp} to each coordinate in $\Pi_\ell\left((\drw)^{-1}\star\rw \right)$, shows that there exists some numbers $K,K' > 0$ such that
\begin{align*}
& \p\left(\left\|\Pi_\ell\left((\drw)^{-1}\star\rw \right)\right\|  \geqslant \delta\right)\leqslant \p\left(\left\|\Pi_\ell\left(\rw -\drw \right) \right\| \geqslant \frac{\delta}{K}\right)\\
		&+ \sum\limits_{k=1}^{K'}\p\left(\left|\omega^k(D_{\frac{1}{n}}S_n, -\drw)P^k(\rw, -\drw)\right| \geqslant \frac{\delta}{K}\right),
\end{align*}
where for each $k \in [K']$, $(\omega^k,P^k)$ are an $\ell$-homogeneous pair. Applying Lemmas \ref{lem:differencebound} and \ref{lem:homoegprob} we see
\[
\p\left(\left\|\Pi_\ell\left((\drw)^{-1}\star\rw \right)\right\|  \geqslant \delta\right) \leqslant (K' + 1)\exp\left(-C\left(\frac{\delta^2}{K^2}m\right)^{\frac{1}{\ell}}n\right).
	\]
	Thus,
	\begin{align*}\lim\limits_{m\to \infty}\limsup\limits_{n\to\infty} &\frac{1}{n}\log\p\left(\left\|\Pi_\ell\left((\drw)^{-1}\star\rw \right)\right\|  \geqslant \delta\right)\\
		&\leqslant \lim\limits_{m\to \infty}\limsup\limits_{n\to\infty} \frac{1}{n}\log\left((K' + 1)\exp\left(-C\left(\frac{\delta^2}{K^2}m\right)^{\frac{1}{\ell}}n\right)\right)\\
		&= \lim\limits_{m\to \infty} -C\left(\frac{\delta^2}{K^2}m\right)^{\frac{1}{\ell}} = -\infty.
	\end{align*}
\end{proof}

\subsection{LDP for the random walk}
\label{ss.group-ldp}

Before proceeding to the proof of Theorem~\ref{t.main}, we introduce the following notation for some particular piecewise linear paths in $G$.
\begin{notation}\label{n.path}
	Fix $m\in\mathbb{N}$, and let $u\in\h^m$. We let $\sigma_{m,u}:[0,1]\to G$ denote the horizontal path such that $\sigma_{m,u}(0)=e$ and
\[
c_{\sigma_{m,u}}(t) = \left(L_{\sigma_{m,u}(t)^{-1}}\right)_{\ast}\sigma^{\prime}_{m,u}(t)=mu_{k} \qquad \text{for } \frac{k-1}{m}< t< \frac{k}{m}
\]
for $k \in[m]$.
\end{notation}

\begin{example}
It is a useful exercise to write these paths explicitly, at least in the step 2 case. We follow the notation in Example \ref{ex.step2}. For a given $m$ and $u=(u_1,\dots, u_m)\in\h^m$, we may use the expression for horizontal paths given by \eqref{e.horiz2} to explicitly describe the path $\sigma=\sigma_{m,u}$. For $0< t < \frac{1}{m}$,
\[	
\sigma(t) = \left(tmu_{1},\int_0^t Q^{(2)}(smu_{1},mu_{1})\,ds \right) = (tmu_{1},0),
\]
for $\frac{1}{m}< t < \frac{2}{m}$, writing $\sigma(t)=(A(t),a(t))$,
	\begin{align*}
		A(t) &= u_{1}+\left(t-\frac{1}{m}\right)mu_{2} \\
		a(t) &= \left(\int_0^{1/m} Q^{(2)}(smu_{1},mu_{1})\,ds
		+ \int_{1/m}^t Q^{(2)}\left(u_{1}+\left(s-\frac{1}{m}\right)mu_{2},mu_{2}\right)\,ds \right) \\
		&= \left(t-\frac{1}{m}\right)Q^{(2)}(u_{1},mu_{2}),
	\end{align*}
and generally for $\frac{k-1}{m}<t < \frac{k}{m}$, $k \in[m]$,
\begin{align*}
	A(t) &= \sum_{j=1}^{k-1} u_j +\left(t-\frac{k-1}{m}\right)mu_{k}\\
	a(t) &=	\sum_{\ell=2}^{k-1} \int_{(\ell-1)/m}^{\ell/m} Q^{(2)}\left(\sum_{j=1}^{\ell-1} u_j +\left(s-\frac{\ell-1}{m}\right)mu_\ell,mu_\ell\right)\,ds  \\
			&\qquad + \int_{(k-1)/m}^t Q^{(2)}\left(\sum_{j=1}^{k-1} u_j +\left(s-\frac{k-1}{m}\right)mu_{k},mu_{k}\right)\,ds \\
		&= \sum_{\ell=2}^{k-1}\sum_{j=1}^{\ell-1}Q^{(2)}(u_j,u_\ell)+\left(t-\frac{k-1}{m}\right)\sum_{j=1}^{k-1}\omega(u_j,mu_{k}).
	\end{align*}
Note in particular that, for $\sigma(1)=(A(1),a(1))$, we have $A(1)=u_{1}+\cdots+u_m$ and
\[ 
a(1) 
= \sum_{\ell=2}^m\sum_{j=1}^{\ell-1} Q^{(2)}\left( u_j,u_\ell\right)
		= \sum_{j=1}^m\sum_{\ell=j+1}^{m}Q^{(2)}\left( u_j,u_\ell\right),\]
and thus $\sigma(1) = \exp(u_1)\star\cdots\star\exp(u_m)$.
\end{example}

We are now ready to prove Theorem~\ref{t.main}, which follows immediately from the following proposition coupled with the exponentially good approximations established in Propositions \ref{p.exp-approx}, \ref{p.exp-approx-higher}, and \ref{prop:mainprop}
\begin{prop}
	Let $\{X_{k}\}_{k=1}^\infty$ be i.i.d.~mean 0 random variables in $\h$, and set
	\[
	\Lambda(\lambda):= \Lambda_X(\lambda) :=\log \E[\exp(\langle\lambda, X_{k}\rangle_\h)].
	\]
	Consider $Y_n^m$, as in Proposition \ref{p.Hldp}. If $\Psi_m(Y_n^m)$ is an exponentially good approximation for $\rw$. Then for
	\[
	S_n:=\exp(X_{1})\star\cdots\star\exp(X_n)
	\]
	the measures $\{\mu_n\}_{n=1}^\infty$ satisfy  a large deviations principle with rate function given by
\[
	J(x) : = \inf\left\{\int_0^1 \Lambda^{\ast}\left(c_\sigma(t)\right)\,dt : \sigma \text{ horizontal with } \sigma(0)=e \text{ and } \sigma(1)=x\right\},
	\]
\end{prop}
\begin{proof}
For a fixed $m\in\mathbb{N}$, recall the map $\Psi_m:\mathcal{H}^m \to G$ defined by \eqref{e.Psim}.

Note that, as the composition of group products with the exponential map, $\Psi_m$ is continuous. Thus, by Theorem~\ref{thm-contract} (the contraction principle) and Proposition \ref{p.Hldp}, for each $m\in\mathbb{N}$, an LDP holds for $\Psi_m(Y_n^m)=S^{m,1}_n\star\cdots \star S^{m,m}_n$ with the rate function
\begin{equation*}
	J_m(x):= I_m(\Psi_m^{-1}(\{x\}))
		:= \inf\{I_m(u): u\in \h^m \text{ and } \Psi_m(u)=x\}
	\end{equation*}
where $I_m$ is as given in \eqref{eq-I_m}. Note that for each $m$ and $u\in\h^m$
\[
I_m(u) = \int_0^1 \Lambda^{\ast}(c_{\sigma_{m,x}}(t))\,dt,
\]
where $\sigma_{m,u}:[0,1]\to G$ is the particular horizontal path introduced in Notation \ref{n.path}. Note also that $\Psi_m(u)=\sigma_{m,u}(1)$. Thus we may write
\[
J_m(x)=  \inf\left\{\int_0^1 \Lambda^{\ast}(c_{\sigma_{m,u}}(t))\,dt : u\in\h^m \text{ and } \sigma_{m,u}(1)=x\right\}.
\]

For $x \in G$, let
\[
\Sigma_x := \left\{\sigma: \sigma \text{ horizontal with } \sigma(0)=e \text{ and } \sigma(1)=x\right\}
\]
and
\[  \Sigma^m_x : = \left\{\sigma_{m,u}: u\in\h^m \text{ with } \sigma_{m,u}(1)=x\right\}.  \]
With this notation
\begin{align*}
	J_m(x)&=  \inf\left\{\int_0^1 \Lambda^{\ast}(c_{\sigma}(t))\,dt : \sigma\in \Sigma_x^m\right\} 
\end{align*}
and
\[
J(x) = \inf\left\{\int_0^1 \Lambda^{\ast}(c_{\sigma}(t))\,dt : \sigma\in \Sigma_x\right\}.
 \]
Since we know that $\Psi_m(Y_n^m)$ is an exponentially good approximation to $\rw$, by Theorem~\ref{t.exp-approx}, we also know that a weak LDP holds for $D_{\frac1n}S_n$ with the rate function
	\[
\sup_{\varepsilon>0} \liminf_{m\to\infty} \inf_{x^{\prime}\in B(x,\varepsilon)} J_m(x^{\prime}).
\]
Thus, we wish to show that this expression is exactly $J(x)$.
Moreover, noting that $J$ is a good rate function, if we further show that for every closed set $F$
\[
\inf_{x\in F} J(x) \leqslant \limsup_{m\to\infty} \inf_{x\in F} J_m(x),
\]
then the full LDP holds for $D_{\frac1n}S_n$ with the rate function $J$. But we get the latter estimate essentially for free, since for all $m$ and $x\in G$, $J(x)\leqslant J_m(x)$
as $\Sigma_x\supset\Sigma^m_x$ for all $m$.

So now note that, since $\varepsilon<\varepsilon^{\prime}$ implies
	\[ \inf_{x^{\prime}\in B(x,\varepsilon)} J_m(x^{\prime})
		\geqslant \inf_{x^{\prime}\in B(x, \varepsilon^{\prime})} J_m(x^{\prime}), \]
we have
\begin{align*}
	\sup_{\varepsilon>0} \liminf_{m\to\infty} \inf_{x^{\prime}\in B(x,\varepsilon)} J_m(x^{\prime})
	&= \,\lim_{\varepsilon\downarrow0} \liminf_{m\to\infty}
		\inf_{x^{\prime}\in B(x,\varepsilon)} J_m(x^{\prime}). 
	\end{align*}
Since $J_m(x^{\prime})\geqslant J(x^{\prime})$ for all $m$ and $x^{\prime}$,
\begin{align*}
\lim_{\varepsilon\downarrow0} \lim_{\ell\to \infty} \inf_{m\geqslant \ell}
		\inf_{x'\in B(x,\varepsilon)} J_m(x')
	&\geqslant \lim_{\varepsilon\downarrow0} \lim_{\ell\to \infty} \inf_{m\geqslant \ell} \inf_{x'\in B(x,\varepsilon)} J(x') \\
	&= \lim_{\varepsilon\downarrow0} \inf_{x'\in B(x,\varepsilon)} J(x')
	= J(x).
	\end{align*}
Thus, if $J(x)=\infty$ then we are done.

So assume that $J(x)<\infty$. Then, given any $\delta>0$, there exists $\gamma\in \Sigma_x$ such that
\begin{equation*}
\int_0^1 \Lambda^{\ast}(c_{\gamma}(t))\,dt < J(x) + \delta/2.
\end{equation*}
The proof will conclude by showing that we can approximate the left hand side arbitrarily well with curves from $\Sigma_{x'}^m$, when $x'$ is close to $x$.
Indeed, below we prove Lemma \ref{l.lip} and Proposition \ref{p.rb}, and together these imply that for all $\varepsilon>0$ and $\ell\in\mathbb{N}$, there exist $m\geqslant\ell$, $x'\in B(x,(C+1)\varepsilon)$ (where $C<\infty$ is the constant in Lemma \ref{l.Gronwall2}), and $\gamma_m\in\Sigma_{x'}^m$ such that
\begin{equation*}
 \left|\int_0^1 \Lambda^{\ast}(c_{\gamma}(t))\,dt
	- \int_0^1 \Lambda^{\ast}(c_{\gamma_m}(t))\,dt \right|<\varepsilon.
\end{equation*}
We then get
\[ \lim_{\varepsilon\downarrow0} \lim_{\ell\to \infty} \inf_{m\geqslant \ell} \inf_{x'\in B(x,\varepsilon)} J_m(x')
	\leqslant J(x)
	\]
which completes the proof.
\end{proof}

\begin{lemma}\label{l.lip}
	Suppose $c_\gamma\in L^1([0,1],\h)$ and $\Lambda^{\ast}(c_\gamma)\in L^1([0,1],[0,\infty))$. Then given $\varepsilon>0$ there exists a horizontal Lipschitz path $\sigma:[0,1]\to G$ so that $\rho_{cc}(\gamma(1),\sigma(1))< C\varepsilon$, where $C<\infty$ is a constant depending on $\gamma$ as in Lemma \ref{l.Gronwall2}, and
\[
\int_0^1 \Lambda^{\ast}(c_\gamma(t))\, dt -\int_0^1 \Lambda^{\ast}(c_{\sigma}(t))\,dt<\varepsilon.
\]
\end{lemma}

\begin{proof}
	Since $c_\gamma\in L^1([0,1],\h)$ and $\Lambda^{\ast}(c_\gamma)\in L^1([0,1],[0,\infty))$, there exists $\delta>0$ such that if $E\subset[0,1]$ is a measurable set with Lebesgue measure $|E|<\delta$, then
	\[ \int_0^1 |c_\gamma(t)|_\h 1_E(t)\,dt < \varepsilon \]
and
	\[ \int_0^1 \Lambda^{\ast}(c_\gamma(t))1_E(t)\,dt < \varepsilon. \]
Also, $c_\gamma\in L^1([0,1],\h)$ implies that there exists $R<\infty$ such that $|\{t\in[0,1]: |c_\gamma(t)|_\h>R\}|<\delta$ by Markov's inequality.

	Let $\sigma:=\gamma_R$ be the continuous path such that $\gamma_R(0)=e$ and
	\[ c_{\gamma_R}(t) = c_\gamma(t)1_{\{|c_\gamma(t)|_\h\leqslant R\}}.\]
Then $\sigma=\gamma_R$ is a horizontal path and
\begin{align*}
	\int_0^1 |c_\gamma(t)- c_{\gamma_R}(t)|_\h\, dt
	= \int_0^1 |c_\gamma(t)| 1_{\{|c_\gamma(t)|_\h>R\}}\,dt < \varepsilon.
\end{align*}
	Thus by Gr\"{o}nwall's Lemma~ \ref{l.Gronwall2}), $\rho(\gamma(1),\gamma_R(1))<C\varepsilon$. We also have that
\begin{align*}
\int_0^1 \Lambda^{\ast}(c_\gamma(t))\, dt -\int_0^1 \Lambda^{\ast}(c_{\gamma_R}(t))\,dt
		&= \int_0^1 (\Lambda^{\ast}(c_\gamma(t))-\Lambda^{\ast}(0))1_{\{|c_\gamma(t)|_\h>R\}}\,dt\\
		&\leqslant \int_0^1 2\Lambda^{\ast}(c_\gamma(t))1_{\{|c_\gamma(t)|_\h>R\}}\,dt < \varepsilon
\end{align*}
	for sufficiently large $R$, since $\Lambda^{\ast}$ is a convex non-negative function and thus for $R$ sufficiently large, $|v|\geqslant R$ implies that $\Lambda^{\ast}(v)\geqslant \Lambda^{\ast}(u)$ for any $|u|\leqslant |v|$.
\end{proof}

\begin{lemma}\label{l.au}
	Suppose $a:[0,T]\to\R^N$ is a bounded measurable function and $\{\pi_m\}$ is a sequence of partitions $\pi_m=\{0=\pi_m^0\leqslant \pi_m^1\leqslant\cdots\leqslant \pi_m^m=T\}$ such that $\operatorname{mesh}\,\pi_m\to0$. Given any $\varepsilon>0$, there exists $m$ sufficiently large, a simple function $a_m: [0,T] \to \R^N$ defined on $\pi_m$, and a measurable subset $E_m\subseteq[0,T]$ such that $|E_m^c|<\varepsilon$, where $|\cdot|$ denotes the Lebesgue measure,  and $|a_m- a|<\varepsilon$ on $E_m$.
\end{lemma}

\begin{proof}
	Without loss of generality we may assume  $T=1$ and $|a|\leqslant 1$. It suffices to consider the case where $N=1$ and $a\geqslant0$. To deal with $a$ taking values in $\R$, we would simply construct the same approximations to the positive and negative parts of $a$, and similarly for $N>1$ we would consider $a$ componentwise.

Fix $k\in\mathbb{N}$ sufficiently large that $\frac{1}{2^k}<\varepsilon$. Let $T^j_k:=a^{-1}\left(\left(\frac{j}{2^k},\frac{j+1}{2^k}\right])\right)$ for $j=0,\ldots,2^k-1$.

Let $U_m^\ell := \left[\pi_m^\ell,\pi_m^{\ell+1}\right)$ for $\ell=0,\ldots,m-1$. Since the $T^k_j$'s are measurable, we may find sufficiently large $m=m(k)$ so that for $j=0,1,\ldots,2^k-1$ there exist disjoint subsets $I^k_j\subset\{1,\ldots, m\}$ which partition $\{1,\ldots, m\}$ such that, for each $j$,
\[
\Delta^j_k:=T_k^j\,\Delta\,\left(\cup_{\ell\in I_k^j} U_m^\ell\right)
\]
has Lebesgue measure as small as one wants, for example, $|\Delta^j_k|<\frac{1}{4^k}$. These can be chosen so that $\ell\in I^j_k$ implies that $U_m^\ell\cap T^j_k\ne\emptyset$. Taking $\Delta_k:=\cup_{j=0}^{2^k-1} \Delta^j_k$, we have that $|\Delta_k| < 2^k\cdot \frac{1}{4^k} =\frac{1}{2^k}$. Take $E_m:=(\Delta_k)^c$.

	For each $\ell=0,\ldots,m-1$, we have $\ell\in I^j_k$ for some $j$ and we can fix some $t_\ell^*\in U_m^\ell\cap T_k^j$, and define
	\[ a_m(t) := \sum_{\ell=0}^{m-1} a(t_\ell^*) \mathbbm{1}_{U_m^\ell}(t). \]
Recalling that $\left|a(t)- \frac{j}{2^k}\right| < \frac{1}{2^k}$ for all $t\in T_j^k$ including $t_\ell^*$,
we have that, for $t\in T^j_k\cap U_m^\ell$,
\begin{align*}
	|a_m(t)-a(t)| &\leqslant \left|a_m(t)-\frac{j}{2^k}\right| + \left|\frac{j}{2^k}-a(t)\right|
		< \left|a(t_\ell^*) - \frac{j}{2^k}\right| + \frac{1}{2^k}
		< 2\cdot\frac{1}{2^k}.
\end{align*}
Thus
	\[ |a_m-a| < 2\varepsilon \qquad \text{ on } E_m = \bigcup_{j=0}^{2^k-1} \left(T^j_k\cap \left(\cup_{\ell\in I^j_k} U_m^\ell\right)\right). \]

\end{proof}

\begin{prop}\label{p.rb}
Suppose $\gamma$ is a horizontal path such that $\gamma$ is Lipschitz, and the Legendre transform satisfies $\Lambda^{\ast}\left( c_{\gamma})\in L^1([0, 1], [0, \infty) \right)$. 	Given $\varepsilon>0$, there exist $m \in \mathbb{N}$, $x^{\prime}\in B(x,\varepsilon)$, and $\gamma_m\in\Sigma_{x^{\prime}}^m$ such that
\begin{equation*}
 \left|\int_0^1 \Lambda^{\ast}(c_{\gamma}(t))\,dt
	- \int_0^1 \Lambda^{\ast}(c_{\gamma_m}(t))\,dt \right|<\varepsilon.
\end{equation*}

\end{prop}

\begin{proof}
Since $\gamma$ is Lipschitz, there exists $R<\infty$ such that $|c_{\gamma}(t)|_\h\leqslant R$ for a.e.~$t$. Since $\Lambda^{\ast}$ is convex on $\h$, $\Lambda^{\ast}$ is Lipschitz continuous on compact subsets. Fix $K\subset\h$ to be the closed ball of radius $2R$ centered at 0, and let $C'=C'_{\Lambda^{\ast}}(R)$ denote the Lipschitz coefficient of $\Lambda^{\ast}$ on $K$.

	By Lemma \ref{l.au}, we may define a sequence of simple functions $\varphi_m$ on the partitions $\pi_m=\{0<\frac{1}{m}<\cdots<\frac{m-1}{m}<1\}$ such that $\varphi_m\to c_\gamma$ in $L^1([0,1],\h)$.  Thus we may choose $m$ sufficiently large that $|c_\gamma-\varphi_m|_{L^1([0,1],\h)}<\delta$, where $\delta<\min\{\varepsilon/C',\varepsilon/C\}$ where $C$ is the constant appearing in the statement of Lemma \ref{l.Gronwall2}. Define $\gamma_m:[0,1]\to G$ to be the horizontal piecewise linear path so that $\gamma_m(0)=e$ and $c_{\gamma_m}(t) =\varphi_m(t)$ for all $t\in[0,1]$ where $\varphi_m$ is continuous. Note that $\gamma_m=\sigma_{m,u}$ where $u=(u_{1},\ldots,u_m)$ is given by $u_{k}=\varphi_m(t)$ for $\frac{k-1}{m}<t<\frac{k}{m}$. By Lemma \ref{l.Gronwall2}, $\gamma_m(1)\in B(x,\varepsilon)$, and
\begin{align*}
	\left|\int_0^1 \Lambda^{\ast}(c_\gamma(t))\,dt - \int_0^1 \Lambda^{\ast}(c_{\gamma_m}(t))\,dt\right|
	&\leqslant \int_0^1 |\Lambda^{\ast}(c_\gamma(t))- \Lambda^{\ast}(c_{\gamma_m}(t))|\,dt \\
	&\leqslant C' \int_0^1 |c_\gamma(t)- c_{\gamma_m}(t)|_\h \,dt < C'\delta< \varepsilon.
\end{align*}
\end{proof}

\begin{remark}
Thus we see that, for the rate function $J$, the infimum is attained over Lipschitz paths.
It is also standard that the horizontal distance is defined taking the infimum over horizontal Lipschitz paths, rather than just over horizontal paths. As in the case of the rate function, these definitions are equivalent which we may see as follows.

As before, we have
\[ d_{cc}(x) = \inf\{\ell(\gamma):\gamma:[0,1]\to G \text{ horizontal and } \gamma(0)=e, \gamma(1)=x\}, \]
and also take
\[d_L(x) := \inf\{\ell(\gamma):\gamma:[0,1]\to G \text{ horizontal, Lipschitz, and } \gamma(0)=e,  \gamma(1)=x\}. \]
Clearly one has $d_L\geqslant d_{cc}$. Now, any horizontal curve of positive length is an absolutely continuous reparameterization of an arclength parameterized horizontal curve, that is, for any horizontal $\gamma:[0,1]\to G$ with $\ell(\gamma)>0$, there exists a (Lipschitz) $\widetilde{\gamma}:[0,\ell(\gamma)]\to G$ with $|c_{\widetilde{\gamma}}|_\h=1$ so that $\gamma=\widetilde{\gamma}\circ\varphi$ for $\varphi:[0,1]\to[0,\ell(\gamma)]$ given by $\varphi(t) =\int_0^t |c_\gamma(s)|\,ds$; see for example Lemma 3.71 of \cite{AgrachevBarilariBoscainBook2019}. So for any horizontal path $\gamma$ with $\ell(\gamma)<\infty$, we have the Lipschitz horizontal path  $\hat{\gamma}:[0,1]\to G$ given by $\hat{\gamma}(t):=\widetilde{\gamma}(\ell(\gamma) t)$ which satisfies $\hat{\gamma}(0)=e$, $\hat{\gamma}(1)=x$. Finally, the length of an absolutely continuous curve is invariant under an absolutely continuous reparameterization (this is just a change of variables in the integral, or see for example Lemma 3.70 of \cite{AgrachevBarilariBoscainBook2019}), and so $\ell(\hat{\gamma})=\ell(\gamma)$. Thus $d_L\leqslant d_{cc}$.

\end{remark}

\subsection{Solving the variational problem for Gaussian random walks}
We now consider the implications of Theorem~\ref{t.main} when $X_n$ are i.i.d.~$\mathcal{N}(0,\mathrm{Id}_{\mathcal{H}})$ random variables on $\h$. In this case, it is straightforward to see $\Lambda(\lambda)=\frac{1}{2}|\lambda|_\h^2$, $\Lambda^{\ast}(u)=\frac{1}{2}|u|_\h^2$, and so Theorem~\ref{t.main} gives an LDP for the associated sub-Riemannian random walk with the rate function

\begin{align*}
	J_{\mathcal{N}}(x) &:= \inf\left\{\frac{1}{2}\int_0^1 |c_\gamma(t)|_\h^2\,dt: \gamma \text{ horizontal, } \gamma(0)=e, \gamma(1)=x\right\} \\
		&\notag= \inf\left\{\frac{1}{2}E(\gamma): \gamma \text{ horizontal, } \gamma(0)=e, \gamma(1)=x\right\},
\end{align*}
where $E(\gamma)$ is the so-called energy of the path $\gamma$. Recall that the length of a horizontal path is given by \eqref{eq-length}. Using an argument similar to \cite[Lemma 3.64]{AgrachevBarilariBoscainBook2019} we have the following lemma.

\begin{lemma}\label{l.3.12}
	A horizontal curve $\gamma:[0,1]\to G$ is a minimizer of $\int_0^1 |c_\gamma(t)|_\h^2\,dt$ among the set of horizontal curves joining $e$ and $x$ if and only if it is a minimizer of the length functional $\ell(\gamma)$ among the horizontal curves joining $e$ and $x$.
\end{lemma}
\begin{proof}
We know that
\[
\ell^2(\gamma) = \left(\int_0^1 | c_\gamma \left( t \right)|_{\mathcal{H}}dt\right)^2\leqslant\int_0^1 | c_\gamma \left( t \right)|^2_{\mathcal{H}}dt,
\]
where equality holds if and only if $ | c_\gamma \left( t \right)|_{\mathcal{H}}$ is constant for all $t\in[0,1]$. The result follows since any horizontal curve is an absolutely continuous reparametrization of an arc length-parameterized horizontal path \cite[Lemma 3.71]{AgrachevBarilariBoscainBook2019} and $\ell(\gamma)$ is invariant under absolutely continuous reparameterizations \cite[Lemma 3.70]{AgrachevBarilariBoscainBook2019}.
\end{proof}

Thus, we can conclude that the minimizer of $J_\mathcal{N}(x)$ (which neccesarily satisfies $|J_\mathcal{N}(x)| <\infty$) is indeed a minimizer of $d_{cc}(x)$. As a nice consequence we obtain Corollary \ref{c.3.13}, a precise LDP for the normal sampling random walk, in terms of the sub-Riemannian metric.

\begin{proof}[Proof of Corollary \ref{c.3.13}]
This is a direct consequence of Chow--Rashevskii's Theorem and Lemma~\ref{l.3.12}.
\end{proof}

\bibliographystyle{amsplain}

\begin{thebibliography}{10}

\bibitem{AgrachevBarilariBoscainBook2019}
Andrei Agrachev, Davide Barilari, and Ugo Boscain, \emph{A comprehensive
  introduction to sub-{R}iemannian geometry}, Cambridge Studies in Advanced
  Mathematics, vol. 181, Cambridge University Press, Cambridge, 2020, From the
  Hamiltonian viewpoint, With an appendix by Igor Zelenko. \MR{3971262}

\bibitem{AgrachevBoscainNeelRizzi2018}
Andrei Agrachev, Ugo Boscain, Robert Neel, and Luca Rizzi, \emph{Intrinsic
  random walks in riemannian and sub-riemannian geometry via volume sampling},
  ESAIM: COCV \textbf{24} (2018), no.~3, 1075--1105.

\bibitem{Augeri2020a}
Fanny Augeri, \emph{Nonlinear large deviation bounds with applications to
  {W}igner matrices and sparse {E}rd{\H{o}}s-{R}\'{e}nyi graphs}, Ann. Probab.
  \textbf{48} (2020), no.~5, 2404--2448. \MR{4152647}

\bibitem{Azencott1980}
R.~Azencott, \emph{Grandes d\'{e}viations et applications}, Eighth {S}aint
  {F}lour {P}robability {S}ummer {S}chool---1978 ({S}aint {F}lour, 1978),
  Lecture Notes in Math., vol. 774, Springer, Berlin, 1980, pp.~1--176.
  \MR{590626}

\bibitem{BaldiCaramellino1999}
Paolo Baldi and Lucia Caramellino, \emph{Large and moderate deviations for
  random walks on nilpotent groups}, J. Theoret. Probab. \textbf{12} (1999),
  no.~3, 779--809. \MR{1702883}

\bibitem{BaudoinFengGordina2019}
Fabrice Baudoin, Qi~Feng, and Maria Gordina, \emph{Integration by parts and
  quasi-invariance for the horizontal wiener measure on foliated compact
  manifolds}, Journal of Functional Analysis \textbf{277} (2019), no.~5, 1362
  -- 1422.

\bibitem{Benard2023}
Timoth\'{e}e B\'{e}nard, \emph{Drift of random walks on abelian covers of
  finite volume homogeneous spaces}, Bull. Soc. Math. France \textbf{151}
  (2023), no.~3, 407--434. \MR{4715578}

\bibitem{BonfiglioliLanconelliUguzzoniBook}
A.~Bonfiglioli, E.~Lanconelli, and F.~Uguzzoni, \emph{Stratified {L}ie groups
  and potential theory for their sub-{L}aplacians}, Springer Monographs in
  Mathematics, Springer, Berlin, 2007. \MR{2363343}

\bibitem{BoscainNeelRizzi2017}
Ugo Boscain, Robert Neel, and Luca Rizzi, \emph{Intrinsic random walks and
  sub-{L}aplacians in sub-{R}iemannian geometry}, Adv. Math. \textbf{314}
  (2017), 124--184. \MR{3658714}

\bibitem{Breuillard2014}
Emmanuel Breuillard, \emph{Geometry of locally compact groups of polynomial
  growth and shape of large balls}, Groups Geom. Dyn. \textbf{8} (2014), no.~3,
  669--732. \MR{3267520}

\bibitem{ChatterjeeDembo2016}
Sourav Chatterjee and Amir Dembo, \emph{Nonlinear large deviations}, Adv. Math.
  \textbf{299} (2016), 396--450. \MR{3519474}

\bibitem{ChatterjeeVaradhan2011}
Sourav Chatterjee and S.~R.~S. Varadhan, \emph{The large deviation principle
  for the {E}rd{\H{o}}s-{R}\'enyi random graph}, European J. Combin.
  \textbf{32} (2011), no.~7, 1000--1017. \MR{2825532}

\bibitem{CorwinGreenleafBook}
Lawrence~J. Corwin and Frederick~P. Greenleaf, \emph{Representations of
  nilpotent {L}ie groups and their applications. {P}art {I}}, Cambridge Studies
  in Advanced Mathematics, vol.~18, Cambridge University Press, Cambridge,
  1990, Basic theory and examples. \MR{1070979 (92b:22007)}

\bibitem{DemboZeitouniBook2010}
Amir Dembo and Ofer Zeitouni, \emph{Large deviations techniques and
  applications}, Stochastic Modelling and Applied Probability, vol.~38,
  Springer-Verlag, Berlin, 2010, Corrected reprint of the second (1998)
  edition. \MR{2571413}

\bibitem{DeuschelStroockBook1989}
Jean-Dominique Deuschel and Daniel~W. Stroock, \emph{Large deviations}, Pure
  and Applied Mathematics, vol. 137, Academic Press, Inc., Boston, MA, 1989.
  \MR{997938}

\bibitem{Eldan2018}
Ronen Eldan, \emph{Gaussian-width gradient complexity, reverse log-{S}obolev
  inequalities and nonlinear large deviations}, Geom. Funct. Anal. \textbf{28}
  (2018), no.~6, 1548--1596. \MR{3881829}

\bibitem{FranchiSerapioniSerra_Cassano2003a}
Bruno Franchi, Raul Serapioni, and Francesco Serra~Cassano, \emph{On the
  structure of finite perimeter sets in step 2 {C}arnot groups}, J. Geom. Anal.
  \textbf{13} (2003), no.~3, 421--466. \MR{1984849}

\bibitem{FranchiSerapioni2016}
Bruno Franchi and Raul~Paolo Serapioni, \emph{Intrinsic {L}ipschitz graphs
  within {C}arnot groups}, J. Geom. Anal. \textbf{26} (2016), no.~3,
  1946--1994. \MR{3511465}

\bibitem{Gaveau1977a}
Bernard Gaveau, \emph{Principe de moindre action, propagation de la chaleur et
  estim\'ees sous elliptiques sur certains groupes nilpotents}, Acta Math.
  \textbf{139} (1977), no.~1-2, 95--153. \MR{0461589 (57 \#1574)}

\bibitem{GordinaLaetsch2017}
Maria Gordina and Thomas Laetsch, \emph{A convergence to {B}rownian motion on
  sub-{R}iemannian manifolds}, Trans. Amer. Math. Soc. \textbf{369} (2017),
  no.~9, 6263--6278. \MR{3660220}

\bibitem{JansonBook1997}
Svante Janson, \emph{Gaussian {H}ilbert spaces}, Cambridge Tracts in
  Mathematics, vol. 129, Cambridge University Press, Cambridge, 1997.
  \MR{MR1474726 (99f:60082)}

\bibitem{JorgensenE1975}
Erik J{\o}rgensen, \emph{The central limit problem for geodesic random walks},
  Z. Wahrscheinlichkeitstheorie und Verw. Gebiete \textbf{32} (1975), 1--64.
  \MR{0400422}

\bibitem{KraaijRedigVersendaal2019}
Richard~C. Kraaij, Frank Redig, and Rik Versendaal, \emph{Classical large
  deviation theorems on complete {R}iemannian manifolds}, Stochastic Process.
  Appl. \textbf{129} (2019), no.~11, 4294--4334. \MR{4013863}

\bibitem{LeDonneNotes2018}
Enrico Le~Donne, \emph{{L}ecture notes on sub-{R}iemannian geometry}, lecture
  notes, 2018.

\bibitem{LubetzkyZhao2015}
Eyal Lubetzky and Yufei Zhao, \emph{On replica symmetry of large deviations in
  random graphs}, Random Structures Algorithms \textbf{47} (2015), no.~1,
  109--146. \MR{3366814}

\bibitem{Melcher2009a}
Tai Melcher, \emph{Heat kernel analysis on semi-infinite {L}ie groups}, J.
  Funct. Anal. \textbf{257} (2009), no.~11, 3552--3592. \MR{2572261
  (2011b:58074)}

\bibitem{Neuenschwander1997}
Daniel Neuenschwander, \emph{Laws of large numbers on simply connected step
  {$2$}-nilpotent {L}ie groups}, Probab. Math. Statist. \textbf{17} (1997),
  no.~1, Acta Univ. Wratislav. No. 1928, 167--177. \MR{1455616}

\bibitem{Pap1999}
G.~Pap, \emph{Lindeberg-{F}eller theorems on {L}ie groups}, Arch. Math. (Basel)
  \textbf{72} (1999), no.~5, 328--336. \MR{1680388 (2000i:60009)}

\bibitem{Pap1993}
Gyula Pap, \emph{Central limit theorems on nilpotent {L}ie groups}, Probab.
  Math. Statist. \textbf{14} (1993), no.~2, 287--312 (1994). \MR{1321768}

\bibitem{Pap1997a}
\bysame, \emph{Construction of processes with stationary independent increments
  in {L}ie groups}, Arch. Math. (Basel) \textbf{69} (1997), no.~2, 146--155.
  \MR{1458701 (98h:60011)}

\bibitem{Sert2019}
Cagri Sert, \emph{Large deviation principle for random matrix products}, Ann.
  Probab. \textbf{47} (2019), no.~3, 1335--1377. \MR{3945748}

\bibitem{VaropoulosBook1992}
N.~Th. Varopoulos, L.~Saloff-Coste, and T.~Coulhon, \emph{Analysis and geometry
  on groups}, Cambridge University Press, Cambridge, 1992. \MR{95f:43008}

\bibitem{Versendaal2019a}
Rik Versendaal, \emph{Large deviations for geodesic random walks}, Electron. J.
  Probab. \textbf{24} (2019), Paper No. 93, 39. \MR{4003146}

\bibitem{Versendaal2019b}
\bysame, \emph{Large deviations for random walks on {L}ie groups}, 2019.

\bibitem{VershyninBook2018}
Roman Vershynin, \emph{High-dimensional probability}, Cambridge Series in
  Statistical and Probabilistic Mathematics, vol.~47, Cambridge University
  Press, Cambridge, 2018, An introduction with applications in data science,
  With a foreword by Sara van de Geer. \MR{3837109}

\end{thebibliography}
\providecommand{\bysame}{\leavevmode\hbox to3em{\hrulefill}\thinspace}
\providecommand{\MR}{\relax\ifhmode\unskip\space\fi MR }
\providecommand{\MRhref}[2]{%
  \href{http://www.ams.org/mathscinet-getitem?mr=#1}{#2}
}
\providecommand{\href}[2]{#2}

\appendix

\section{Proof of the upper bound}
This section provides the proof of the upper bound \eqref{e.prop-ldp-hm}. For fixed $m$, let $\{Y_n^m\}_{n\geqslant 1}$ be as described in Section~\ref{s.LDP}. The following result arises naturally in the usual proof of the upper bound of Cram\'{e}r's Theorem. Indeed, here we follow the proof in \cite{DemboZeitouniBook2010}. However, the expression we leave for the upper bound is perhaps not completely standard. So for clarity we provide the proof here.
\begin{prop}\label{prop-ldp-hm}
	Fix $m\in\mathbb{N}$. For any closed $F\subset \mathcal{H}^m$ we have that
\[
\limsup_{n \to \infty}\frac1n \log\p \left(Y_n^{m}\in F \right)
\leqslant - \inf_{x\in F} I_m(x)
\]
where
	\[I_m(x)= \sup_{\lambda\in \mathcal{H}^m} \left\{ \langle \lambda, x\rangle_{\mathcal{H}^m}  - \limsup_{n\to\infty}\frac1n \log \E \left[\exp\left(n\langle \lambda, Y_n^{m}\rangle_{\mathcal{H}^m} \right)\right] \right\}. \]
\end{prop}

\begin{proof} We use the argument for proving the upper bound in the classical Cram\'{e}r's theorem, see for example \cite[p. 37]{DemboZeitouniBook2010}. First we assume that $F$ is compact, and note that it suffices to prove that, for any $\delta>0$,
\begin{align*}
& \limsup_{n \to \infty}\frac1n \log\p \left(Y_n^{m}\in F \right)
\leqslant \delta - \inf_{x\in F} I_m^\delta(x)
\end{align*}
where $I_m^\delta(x):=\min\left\{\frac1\delta, I_m(x)-\delta\right\}$.

and
\[
 I_m(x)= \sup_{\lambda\in \mathcal{H}^m} \left\{ \langle \lambda, x\rangle_{\mathcal{H}^m}  - \limsup_{n\to\infty}\frac1n \log \E \left(\exp\left(n\langle \lambda, Y_n^{m}\rangle_{\mathcal{H}^m} \right)\right) \right\}
\]

	So fix $\delta>0$ and let $\widetilde\Lambda(\lambda):=\limsup_{n\to\infty}\frac1n \log \E \left[\exp\left(n\langle \lambda, Y_n^{m}\rangle_{\mathcal{H}^m} \right)\right]$.
For every $q\in F$, let $\lambda_q\in \mathcal{H}^m$  be such that
\begin{equation}\label{eq-lambda-q}
\langle \lambda_q,q\rangle_{\mathcal{H}^m}-\widetilde\Lambda(\lambda_q)\geqslant I_m^\delta(q),
\end{equation}
and choose $r_q>0$ such that $r_q|\lambda_q|_{\mathcal{H}^m}\le\delta$. Denote by $B_q$ the ball centered at $q$ of radius $r_q$.  By Markov's inequality
\begin{align*}
\E \left[\exp\left(n\langle \lambda_q, Y_n^{m}\rangle_{\h^m} \right) \right]
	&\geqslant \E \left[\exp\left(n\langle \lambda_q, Y_n^{m}\rangle_{\h^m} \right)
		\mathbbm{1}_{Y_n^{m}\in B_q}\right]\\
	&\geqslant \exp\left[n\inf_{x\in B_q}\left\{\langle \lambda_q, x\rangle_{\h^m} \right\} \right]
		\p \left(Y_n^{ m}\in B_q \right).
\end{align*}
Thus we have
\begin{align*}
\p \left(Y_n^{ m}\in B_q \right)\le
	\E \left[\exp\left(n\langle \lambda_q, Y_n^{m}\rangle_{\h^m}
		- n\inf_{x\in B_q}\left\{\langle \lambda_q, x\rangle_{\h^m} \right\} \right) \right].
\end{align*}
Moreover,
\[
-\inf_{x\in B_q}\{\langle \lambda_q, x\rangle_{\h^m} \} \leqslant r_q|\lambda_q|_{\h^m}-\langle \lambda_q, q\rangle_{\h^m}\leqslant \delta-\langle \lambda_q, q\rangle_{\h^m}.
\]
Thus, for any $q\in F$,
\begin{align*}
\frac1n \log\p \left(Y_n^{ m}\in B_q \right)
	&\leqslant  -\inf_{x\in B_q}\{\langle \lambda_q, x\rangle_{\h^m} \} +\frac1n \log\E \left(\exp\left[n\langle \lambda_q, Y_n^{m}\rangle_{\h^m} \right) \right],
\end{align*}
hence
\[
\limsup_{n \to \infty} \frac1n \log\p \left(Y_n^{ m}\in B_q \right) \leqslant  \delta-\langle \lambda_q, q\rangle_{\h^m}+\widetilde{\Lambda}(\lambda_q).
\]
Now by compactness of $F$ we can extract a finite covering $\cup_{j=1}^N B_{q_j}$, and hence
\begin{align*}
\limsup_{n \to \infty}\frac1n \log\p \left(Y_n^{ m}\in F \right)\leqslant \frac1n \log N+\delta-\min_{1\leqslant j\leqslant N}\{ \langle \lambda_{q_j}, q_j\rangle_{\h^m}-\widetilde{\Lambda}(\lambda_{q_j})\}
\end{align*}
By \eqref{eq-lambda-q} we obtain that
\begin{align*}
& \limsup_{n \to \infty}\frac1n \log\p \left(Y_n^{m}\in F \right)\le
\delta-\min_{1\leqslant j\leqslant N} I_m^\delta(q_j)\leqslant \delta-\inf_{x\in F} I_m^\delta(x).
\end{align*}
Let $\delta\to 0$ we then have
\begin{align*}
& \limsup_{n \to \infty}\frac1n \log\p \left(Y_n^{m}\in F \right)\le
-\inf_{x\in F} I_m(x).
\end{align*}
At last we extend the above upper bound to any closed sets $F\subset \h^m$ by using the fact that $\p(Y_n^m\in \cdot)$ is an exponentially tight family of probability measures and by \cite[Lemma 1.2.18]{DemboZeitouniBook2010}. Let $H_r:=[-r,r]^{\dH m}$, then $H^c_r=\cup_{j=1}^{\dH m}\{x:|x_j|>r\}$, hence
\begin{align*}
\p(Y_n^m\in H^c_r)&\leqslant \sum_{j=1}^m \sum_{i=1}^\dH \p(Y_n^{m,j,i}\geqslant r)+ \p(Y_n^{m,j,i}\leqslant -r),
\end{align*}
where $Y_n^{m,j,i}$ represent the coordinates of $Y_n^{m,j}$, that is, $Y_n^{m,j}=(Y_n^{m,j,1},\dots,Y_n^{m,j,\dH})\in \h$.
Again by the Markov inequality we know that
\[
\limsup_{n \to \infty}\frac1n\log \p(Y_n^{m,j,i}\geqslant r)\leqslant -\Lambda^{\ast}(r),\quad\limsup_{n \to \infty} \frac1n\log \p(Y_n^{m,j,i}\leqslant -r)\leqslant -\Lambda^{\ast}(-r),
\]
where $\Lambda^{\ast}(r)=\sup_{\lambda\in \R} \{r\lambda- \limsup_{n\to\infty}\log \E(\exp(n\lambda Y^{m,j,i}))\}$, which tends to $\infty$ as $r\to\infty$.
Hence we have that
\[
\lim_{r\to\infty}\limsup_{n\to\infty}\frac1n\log \p(Y_n^m\in H^c_r)=-\infty.
\]
Since for any closed $F\subset \h^m$, we have
\[
\p(Y_n^m\in F)\leqslant \p(Y_n^m\in F\cap H_r)+\p(Y_n^m\in H^c_r),
\]
and hence
\begin{align*}
\limsup_{n\to\infty}\frac1n\log \p(Y_n^m\in F)&\leqslant \lim_{r\to \infty}\limsup_{n\to\infty}\frac1n\log \p(Y_n^m\in F\cap H_r)\\
&\leqslant -\lim_{r\to \infty}\inf_{x\in F\cap H_r} I_m(x)=-\inf_{x\in F} I_m(x).
\end{align*}
\end{proof}

\section{Random symplectic form}\label{a.RSymplF}
Suppose $\omega$ is a symplectic form on $\mathbb{R}^{N}$ and define

\begin{equation*}
\omega_{ij}:=\omega\left( v^{i}, v^{j} \right)
\end{equation*}
for any $v^{i}, v^{j} \in \mathbb{R}^{N}$.  By antisymmetry of the symplectic form $\omega$ we see that for $k\leqslant \ell \leqslant n$ and $\mathbf{v}=\left\{ v^{i} \right\}_{i=1}^{n}$, $v^{i} \in \mathbb{R}^{N}$ we have

\begin{equation*}
W_{k, \ell}\left( \mathbf{v} \right):=\omega\left( \sum_{i=1}^{k}v^{i}, \sum_{j=1}^{\ell}v^{j}\right)=
\sum_{i=1}^{k}\sum_{j=i+1}^{\ell}
\omega_{ij}.
\end{equation*}
Suppose now that $v^{1}=\left( x^{1}, y^{1} \right), \ldots , v^{n}=\left( x^{n}, y^{n} \right)$ is a collection of random variables  in $\mathbb{R}^{N} \times \mathbb{R}^{N}$, then we can consider the distribution of the \emph{random symplectic form} $W_{k, \ell}$.  Let $k \leqslant \ell \leqslant n$, and $u_{k}^{i}=\left( x_{k}^{i}, y_{k}^{i} \right)$, $k\in [N]$, $i\in[n]$, then
\begin{align*}
& W_{k,l}=\sum_{i=k+1}^{\ell}\sum_{j=i+1}^{\ell}\omega_{ij}=X^{T} A Y, k \leqslant \ell, \text{ where }  \nonumber
\\
& X=\left( x_{1}, \ldots , x_{n} \right), Y=\left( y_{1}, \ldots , y_{n} \right), \nonumber
\\
& A=\left\{ a_{ij} \right\}_{i, j =1}^{n} \text{ with } -a_{ji}=a_{ij}=1 \text{ for }  k+1 \leqslant i <j \leqslant \ell \text{ and } 0 \text{ otherwise},
\end{align*}
so $A$ is a skew-symmetric Toeplitz matrix.  Observe that the Hilbert-Schmidt norm and operator norms of $A$ are
\begin{align}\label{e.Wnorms}
& \Vert A \Vert_{HS}=\left( \sum_{i=k+1}^{\ell}\sum_{j=i+1}^{\ell}\vert a_{ij} \vert^{2}\right)^{1/2}=\sqrt{\left( \ell-k+1\right)\left( \ell-k\right)},\notag
\\
& \Vert A \Vert \leqslant  \ell-k.
\end{align}

Before we describe the distribution of $W_{k, \ell}$, we recall  the definition of a sub-Gaussian random variable. Several equivalent characterizations of sub-Gaussian random variables are given in \cite[Proposition 2.5.2]{VershyninBook2018}. We denote the sub-Gaussian norm (Orlicz norm) by
\[
 \|X\|_{\psi_{2}}:=\inf\left\{s>0: \mathbb{E} e^{(X/s)^{2}}-1\leqslant 1 \right\}.
\]
We say that $X$  is a (centered) sub-Gaussian random variables, if $\mathbb{E} X=0$, and the moment generating function of $X$ satisfies
\[
\mathbb{E} \exp\left( \lambda X \right) \leqslant \exp\left( C \lambda^{2} \|X\|_{\psi_{2}}^{2} \right)
\]
for any $\lambda \in \mathbb{R}$ and some absolute constant $C>0$.
If $X$ and $Y$ are i.i.d. sub-Gaussian random variables, then by \cite[Lemma 2.7.7]{VershyninBook2018} $XY$ is sub-exponential with
\[
 \|XY\|_{\psi_{1}}:=\inf\left\{s>0: \mathbb{E} e^{\frac{\vert XY \vert}{s}}\leqslant 2 \right\} \leqslant \|X\|_{\psi_{2}}\|Y\|_{\psi_{2}}.
\]
This means that by \cite[Proposition 2.7.1]{VershyninBook2018}	there is a constant $K >0$ such that
\[
\mathbb{E} e^{\lambda \vert XY \vert}\leqslant \exp\left( K \lambda \right)
\]
for any $0 \leqslant \lambda \leqslant 1/K$.

\begin{prop}[Properties of random symplectic forms]\label{l.Wij}
Let $\left( x_{n}, y_{n} \right)_{n=1}^{\infty},$ be a sequence of $\mathbb{R}^{2}$-valued i.i.d. random variables with mean zero and  variance $\left( 1, 1 \right)$.

(1) Suppose $X=\left( x_{1}, \ldots , x_{n} \right)$ and $Y=\left( y_{1}, \ldots , y_{n} \right)$ are i.i.d. standard normal random variables. Then the  random variable $W_{ k, \ell}$ is distributed as a linear combination  $\chi$-square distributed variables with coefficients depending on the eigenvalues of the matrix $W$. In addition, there are universal constants $c>0$ and $C>0$ such that
\[
\mathbb{E} e^{ \lambda W_{k, \ell}} \leqslant e^{ C \lambda^{2} \left( \ell-k+1\right)\left( \ell-k\right)} 	
\]
for all $\lambda$ satisfying $\vert \lambda \vert \leqslant c$.

(2) Suppose $X=\left( x_{1}, \ldots , x_{n} \right)$ and $Y=\left( y_{1}, \ldots , y_{n} \right)$ are i.i.d. sub-Gaussian random variables. Then there is a universal constant $D>0$ such that
\begin{align*}
& \mathbb{E} e^{ \lambda W_{k, \ell}}  \leqslant
e^{ C
D^{2} \Vert X \Vert_{\psi_{2}}^{4} \lambda^{2}
 \left( \ell-k+1\right)\left( \ell-k\right)}	
\end{align*}
for all
\[
\vert \lambda \vert \leqslant \frac{c}{D\Vert X \Vert_{\psi_{2}}^{2}}.
\]
Here $C$ and $c$ are the universal constants in (1).
\end{prop}

\begin{proof}
First observe that there is an orthogonal matrix $R$ such that $R^{T}AR$ is block diagonal with blocks being
\[
\left(
  \begin{array}{cc}
    0 & \theta \\
    -\theta & 0 \\
  \end{array}
\right).
\]
Note that $\pm i \theta$ are eigenvalues of $A$, and  they can be found as roots of the characteristic polynomial. We denote by $B$ the matrix

\[
B=\left\{ b_{ij} \right\}_{i, j =1}^{N}, -b_{ji}=b_{ij}=
1 \text{ for }  1 \leqslant i <j \leqslant N \text{ and } 0 \text{ otherwise},
\]
whose characteristic polynomial is
\begin{align*}
& p_{2N}\left( \lambda \right)=\sum_{l=0}^{N}\binom{2N}{2l} \lambda^{2l},
\\
& p_{2N+1}\left( \lambda \right)=-\sum_{l=0}^{N}\binom{2N+1}{2l+1} \lambda^{2l+1},
\end{align*}
since $p_{N+1}\left( \lambda \right)=-\left( \lambda -1 \right)p_{N}\left( \lambda \right)+\left( \lambda +1 \right)^{N}$.

Now we assume that $X=\left( x_{1}, \ldots , x_{n} \right)$ and $Y=\left( y_{1}, \ldots , y_{n} \right)$ are i.i.d. standard normal random variables. Recall that the original matrix $A$ has a copy of a matrix $B$ as a block on the diagonal with all other entries being $0$. Note that the distribution of random vectors $X$ and $Y$ is invariant under orthogonal transformations, so we only need to determinate the distribution $X^{T} A Y$ when $A$ is block diagonal. Moreover, it is enough to look at one block
\[
\left( x_{1}, x_{2} \right)\left(
  \begin{array}{cc}
    0 & \theta \\
    -\theta & 0 \\
  \end{array}
\right)
\left( y_{1}, y_{2} \right)^{T}=\theta\left( x_{1}y_{2}-x_{2}y_{1} \right),
\]
where $x_{1}, x_{2}, y_{1}, y_{2}$ are i.i.d. standard normal random variables.

First consider $U$ and $V$ which are i.i.d. standard normal random variables. Then $U-V$ and $U+V$ are two independent variables distributed as $\mathcal{N}\left( 0, 2 \right)$, and
\[
UV=\frac{\left( U+V \right)^{2}}{4}-\frac{\left( U-V \right)^{2}}{4}
\]
thus $UV$ is distributed as the difference of two $\chi$-square distributed variables. This has a variance-gamma distribution which can be shown by using the moment generating functions. Thus the non-zero components of $X^{T} A Y$ are linear combinations of two $\chi$-square distributed variables.

To prove the estimate we use \cite[Lemma 6.2.2]{VershyninBook2018} giving an estimate of the moment generating function of Gaussian chaos. Namely, if $X$ and $Y$ are independent standard Gaussian vectors, $A$ is a matrix, then
\[
\mathbb{E} \exp\left( \lambda X^{T} A Y  \right) \leqslant  \exp\left(  C \lambda^{2} \Vert A \Vert_{HS}^{2} \right)
\]
for all $\lambda$ satisfying $\vert \lambda \vert \leqslant \frac{c}{\Vert A \Vert}$. Here the constant $C$ is a universal constant for a product of standard Gaussian variables and $c$ is a universal constant for sub-exponential variables. Thus by \eqref{e.Wnorms}
\[
\mathbb{E} \exp\left( \lambda W_{k, \ell}  \right) \leqslant  \exp\left(  C \lambda^{2} \left( \ell-k+1\right)\left( \ell-k\right) \right) 	
\]
for all $\lambda$ satisfying $\vert \lambda \vert \leqslant c$. Note that we can get a similar result from the previous argument as $\chi$-square distributed random variables are sub-exponential and the constants in \cite[Lemma 6.2.2]{VershyninBook2018} use singular numbers of the matrix $A$ which can be estimated by using the Hilbert-Schmidt norm.

To prove the second part of the statement we will use  \cite[Lemma 6.2.3]{VershyninBook2018}  which allows a comparison of the moment generating functions of sub-Gaussian and Gaussian random variables. Suppose $X=\left( x_{1}, \ldots , x_{n} \right)$ and $Y=\left( y_{1}, \ldots , y_{n} \right)$ are i.i.d. sub-Gaussian random variables, $A$ is a matrix, and $U$ and $V$ are two independent standard Gaussian vectors. Denote $K:=\Vert X \Vert_{\psi_{2}}=\Vert Y \Vert_{\psi_{2}}$, then by \cite[Lemma 6.2.3]{VershyninBook2018} there is a universal constant $D>0$ such that
\[
\mathbb{E} \exp\left( \lambda X^{T} A Y  \right) \leqslant  \exp\left(  D K^{2} \lambda U^{T} A V \right)
\]
for any $\lambda  \in \mathbb{R}$. Then we can use the first part applied to $U^{T} A V$ and $\widetilde{\lambda}= D \Vert X \Vert_{\psi_{2}}^{2} \lambda$ to see that
\begin{align*}
& \mathbb{E} \exp\left( \lambda X^{T} A Y  \right) \leqslant \exp\left(  D \Vert X \Vert_{\psi_{2}}^{2} \lambda U^{T} A V \right)
\\
& \leqslant  \exp\left(  C \widetilde{\lambda}^{2} \left( \ell-k+1\right)\left( \ell-k\right) \right) =\exp\left(  C
D^{2} \Vert X \Vert_{\psi_{2}}^{4} \lambda^{2}
 \left( \ell-k+1\right)\left( \ell-k\right) \right)	
\end{align*}
for all $\vert\widetilde{\lambda} \vert \leqslant c$, that is, $\vert \lambda \vert \leqslant c/D\Vert X \Vert_{\psi_{2}}^{2}$.
\end{proof}

\end{document}